\documentclass[pdflatex]{sn-jnl}

\usepackage{amsmath,amssymb,amsfonts}%
\usepackage{amsthm}%
\usepackage{anyfontsize}%
\usepackage{enumitem}%

\numberwithin{equation}{section}%
\newtheorem{theorem}{Theorem}[section]%
\newtheorem{proposition}[theorem]{Proposition}%
\newtheorem{remark}[theorem]{Remark}%
\newtheorem{corollary}[theorem]{Corollary}%

\begin{document}

\title{The error of Chebyshev approximations on shrinking domains}
\author{Tobias Jawecki\footnote{Email: tobias.jawecki@gmail.com}}


\abstract{Previous works show convergence of rational Chebyshev approximants to the Pad\'e approximant as the underlying domain of approximation shrinks to the origin. In the present work, the asymptotic error and interpolation properties of rational Chebyshev approximants are studied in such settings. Namely, the point-wise error of Chebyshev approximants is shown to approach a Chebyshev polynomial multiplied by the asymptotically leading order term of the error of the Pad\'e approximant, and similar results hold true for the uniform error and Chebyshev constants. Moreover, rational Chebyshev approximants are shown to attain interpolation nodes which approach scaled Chebyshev nodes in the limit. Main results are formulated for interpolatory best approximations and apply for complex Chebyshev approximation as well as real Chebyshev approximation to real functions and unitary best approximation to the exponential function.}

\maketitle

\section{Introduction and overview}

In the present work we consider rational approximations to complex functions $f$ defined on $E\subset\mathbb{C}$, where $E$ refers to a given domain which is open and contains the origin. We assume that $f:E\to \mathbb{C}$ is holomorphic on $E$.
Moreover, we let $K\subset E$ denote the domain of approximation, assuming that $K$ is a compact set which has no isolated points.
The most relevant cases might be approximations on the interval $[-1,1]$ or the unit disk $\Delta := \{ z\in\mathbb{C}~|~ |z|\leq 1\}$, i.e., $K=[-1,1]$ or $K=\Delta$.

We let $\mathcal{R}_{mn}$ denote the set of rational functions of the form $r=p/q$ where $p$ and $q$ are complex polynomials of degree $\leq m$ and $\leq n$, respectively, and $q\not\equiv 0$. We may also refer $r\in\mathcal{R}_{mn}$ as  
an $(m,n)$-rational function.
For a complex function $g: K\to \mathbb{C}$ the uniform norm on $K$ is defined as
\begin{equation*}
\|g\|_K:= \max_{z\in K}|g(z)|.
\end{equation*}
We refer to $\widehat{r}\in\mathcal{R}_{mn}$ as a {\em Chebyshev approximant} to $f$ on $K$ if $\widehat{r}$ minimizes the uniform error
\begin{equation*}
\|\widehat{r} - f\|_{K}
= \min_{r\in\mathcal{R}_{mn}} \|r - f\|_{K}.
\end{equation*}
Previous works, most importantly \cite[Theorem~III]{Wa31}, show existence of Chebyshev approximants in the present setting.
We remark that rational Chebyshev approximants in a complex setting may not be unique, cf.~\cite{SV78,GT83}.

The main focus in the present work is on approximations on the scaled set $\varepsilon K = \{\varepsilon z\in\mathbb{C}~|~z\in K\}$, $\varepsilon>0$. In this context, the norm $\|g\|_{\varepsilon K}$ corresponds to the uniform norm on the set $\varepsilon K$, in particular, $\|g\|_{\varepsilon K}=\|g(\varepsilon \cdot)\|_{K}$.

\smallskip
{\em Interpolatory best approximations.}
While our main focus is on Chebyshev approximants, our results apply in a more general setting. Namely, for interpolatory best approximants, i.e., rational functions which minimize a uniform error in a set of rational interpolants. In the present work we show that complex Chebyshev approximants attain interpolation nodes on a disk assuming the underlying domain of approximation is sufficiently small, and thus, on shrinking domains Chebyshev approximation can be understood as interpolatory best approximation. On the other hand, the setting of interpolatory best approximations also includes real Chebyshev approximation to real functions on the interval and unitary best approximation to the exponential function introduced in~\cite{JS23u}.

{\em Convergence and asymptotic error.}
Convergence of rational Chebyshev approximants on $\varepsilon K$ to the Pad\'e approximant as $\varepsilon \to 0$ goes back to~\cite[Theorem~3b]{TG85}. Similar results appeared earlier in~\cite{Wa34,Wa64,CSS74,Wa74,GT83}. Convergence of rational Chebyshev approximants to the Pad\'e approximant on shrinking domains does not imply that the asymptotic error of these approximations coincides in general~\cite[Section~3]{TG85}. Results concerning the asymptotic error of Chebyshev approximation on a shrinking interval appeared earlier in~\cite[Theorem~103]{Me67} with a reference to~\cite{Me6x} for the proof of this theorem. For the asymptotic error of real polynomial Chebyshev approximation ($n=0$) on the interval, we also refer to~\cite{MW60a,Ni62} (summarized in~\cite[Theorem~62]{Me67}).

In the present work we show that similar results hold true for the asymptotic error of interpolatory best approximations (including rational Chebyshev approximation) on arbitrary compact domains in the complex plane. In particular, we show that the point-wise error of interpolatory best approximants approaches a Chebyshev polynomial of $K$ multiplied by the asymptotically leading order term of the error of the Pad\'e approximant. Respectively, the uniform error scales with a Chebyshev constant of $K$ in the limit. Moreover, we show that a rational approximation based on interpolation using scaled Chebyshev nodes as interpolation nodes attains the same asymptotic error as rational Chebyshev approximation.

The main focus of the present work is on the asymptotic error of interpolatory best approximations on shrinking domains for a fixed degree $m$ and $n$. The asymptotic error of Chebyshev approximation on a fixed domain with increasing degree, i.e., $m+n\to \infty$, is discussed in~\cite{Wa69,LS06b,Tre23} and others. While for the limit $m+n\to \infty$ most results focus on the rate of geometric convergence of the error, more explicit representations for the asymptotic error are available when considering the limit $\varepsilon\to 0$ as done in the present work.

For the $(m,n)$-Chebyshev approximation to the exponential function $f=\exp$ on a fixed disk and interval, the asymptotic error for $m+n\to\infty$ is given in explicit form in~\cite{Tre84} and~\cite{Bra84}, respectively. In particular, the uniform error in the limit $m+n\to \infty$ for the interval $[-1,1]$ (unit disk $\Delta$) and in the limit $\varepsilon\to 0$ for fixed degrees $m$ and $n$ and the domain $\varepsilon K$ with $K=[-1,1]$ ($K=\Delta$) attains the same asymptotically leading order term (cf.\ Subsection~\ref{subsec:exp}). It is not clear to the authors of the present work if a similar relation holds true in general. 

\subsection*{Outline}
In Section~\ref{sec:interpolation} we recall some classical results considering rational interpolation and Pad\'e approximation, and in Section~\ref{sec:continuity} we provide some auxiliary results related to continuity of rational interpolants on the nodes. The asymptotic error of rational interpolants with scaled interpolation nodes is provided in Proposition~\ref{prop:convinterp} (Section~\ref{sec:asymerrinterp}). In Section~\ref{sec:chebyshevnodes} we recall classical properties of Chebyshev nodes, and we remark that asymptotic error for rational interpolants with scaled nodes is optimal for Chebyshev nodes.
In Section~\ref{sec:ibest} we introduce interpolatory best approximations.
In Proposition~\ref{prop:Chebinterpolates} (Section~\ref{sec:chebipnodes}) we show that Chebyshev approximants attain interpolation nodes, assuming the domain of approximation is sufficiently small, and thus, Chebyshev approximants can be understood as interpolatory best approximants.
The asymptotic error of interpolatory best approximations on shrinking domains, as well as convergence properties of the underlying interpolation nodes, are provided in Theorem~\ref{thm:asymresults} (Section~\ref{sec:asymerrorbest}).
Some applications of the presented results, including relations to previous asymptotic error representations, are discussed in Section~\ref{sec:consequences}.

\section{Rational interpolation}\label{sec:interpolation}

In the present section we summarize some classical results regarding rational interpolation, which is also referred to as Newton--Pad\'e approximation, multipoint Pad\'e approximation, or osculatory rational interpolation in the literature. For details we refer to~\cite{Wu75, GJ76, Cla78, Gu90} and others. In particular, these references also consider the case that the underlying interpolation nodes may have higher multiplicity which is essential for the present work.
For the simpler case of having distinct interpolation nodes we also refer to~\cite{MW60b} and~\cite[Section~2]{Be70} for an overview. A review on rational interpolation, focusing on Pad\'e approximation, is also available in~\cite[Subsection~7.1]{BG96}.

In the following, we consider interpolants $r\approx f$ with $r\in\mathcal{R}_{mn}$ and $m+n+1$ interpolation nodes $z_0,\ldots,z_{m+n}\in E$ allowing nodes of higher multiplicity.

{\em Divided differences.} To proceed we first recall the definition of divided differences (cf.~\cite{Bo05} or \cite[Subsection~B.16]{Hi08}). 
While divided differences have various equivalent definitions, it is sufficient for the present work to consider its definition in terms of contour integration as for instance in \cite[eq.~(B.29)]{Hi08}. For a complex function $g:E\to \mathbb{C}$ which is holomorphic on $E$, the divided differences at nodes $z_0,\ldots,z_j\in E$, $j\geq 0$, corresponds to
\begin{equation}\label{eq:defdd}
g[z_0,\ldots,z_j] := \frac{1}{2\pi \mathrm{i}} \int_\Gamma \frac{g(\zeta)}{(\zeta-z_0)\cdots(\zeta-z_j)} \,\mathrm{d}\zeta,~~~j\geq 0,
\end{equation}
where $\Gamma$ denotes a closed contour in $E$ that encloses $z_0,\ldots,z_j$.

We make use of the following auxiliary result further below in the present work.
\begin{proposition}\label{prop:ddcont}
Let $\{z_{k,\ell}\}_{\ell\in\mathbb{N}}$ denote a given sequence of nodes with $z_{k,\ell}\to z_k$ for $\ell \to \infty$ and $k=0,\ldots,j$. Let $\{g_{\ell}\}_{\ell\in\mathbb{N}}$ denote a sequence of holomorphic functions $g_{\ell}:E\to \mathbb{C}$.
Assume $\|g_{\ell} -g \|_{K'}\to 0$ for $\ell\to \infty$ where $K'\subset E$ denotes a compact and connected set which contains the nodes $z_k$ and $z_{k,\ell}$ for $k=0,\ldots,j$ and $\ell \in\mathbb{N}$ in its interior.
Then
\begin{equation}\label{eq:convergencegdd}
g_\ell[z_{0,\ell},\ldots,z_{j,\ell}] = g[z_0,\ldots,z_j]
+ \mathcal{O}(\|g_{\ell} -g \|_{K'}) + \sum_{k=0}^{j}\mathcal{O}(|z_{k,\ell} - z_k|),~~\ell\to \infty.
\end{equation}
\end{proposition}
\begin{proof} Consider the divided differences $g_\ell[z_{0,\ell},\ldots,z_{j,\ell}]$ and $g[z_0,\ldots,z_j]$ using the representation~\eqref{eq:defdd} for a closed contour $\Gamma$ in $K'$ which encloses the nodes $z_k$ and $z_{k,\ell}$ for $k=0,\ldots,j$ and $\ell \in\mathbb{N}$. In particular
\begin{subequations}\label{eq:1overprodzeta2term1}
\begin{align}
&|g_\ell[z_{0,\ell},\ldots,z_{j,\ell}] - g[z_0,\ldots,z_j]|\notag\\
&\qquad\qquad \leq \frac{1}{2\pi} \int_\Gamma \left|\frac{g_\ell(\zeta)}{(\zeta-z_{0,\ell})\cdots(\zeta-z_{j,\ell})} - \frac{g(\zeta)}{(\zeta-z_0)\cdots(\zeta-z_j)} \right|\,\mathrm{d}\zeta\notag\\
&\qquad\qquad\leq \frac{\|g_{\ell} -g \|_{K'}}{2\pi} \int_\Gamma \left|\frac{1}{(\zeta-z_{0,\ell})\cdots(\zeta-z_{j,\ell})} \right|\,\mathrm{d}\zeta\label{eq:1overprodzeta1}\\
&\qquad\qquad\quad + \frac{\|g\|_{K'}}{2\pi} \int_\Gamma \left| \frac{(\zeta-z_0)\cdots(\zeta-z_j)-(\zeta-z_{0,\ell})\cdots(\zeta-z_{j,\ell})}{(\zeta-z_{0,\ell})\cdots(\zeta-z_{j,\ell})(\zeta-z_0)\cdots(\zeta-z_j)} \right|\,\mathrm{d}\zeta.\label{eq:1overprodzeta2b}
\end{align}
\end{subequations}
The numerator in the integral in~\eqref{eq:1overprodzeta2b} expands to
\begin{align*}
&(\zeta-z_0)\cdots(\zeta-z_j)-(\zeta-z_{0,\ell})\cdots(\zeta-z_{j,\ell})\\
&\qquad\qquad = \sum_{k=0}^j \left((z_{k,\ell}-z_k) \prod_{i_1=0}^{k-1}(\zeta-z_{i_1})\prod_{i_2=k+1}^{j}(\zeta-z_{i_2,\ell})\right),
\end{align*}
where products over empty index sets correspond to $\prod_{i_1=0}^{-1}=\prod_{i_2=j+1}^{j}=1$.
Thus,
\begin{subequations}\label{eq:1overprodzeta2term2}
\begin{align}
&\int_\Gamma \left| \frac{(\zeta-z_0)\cdots(\zeta-z_j)-(\zeta-z_{0,\ell})\cdots(\zeta-z_{j,\ell})}{(\zeta-z_{0,\ell})\cdots(\zeta-z_{j,\ell})(\zeta-z_0)\cdots(\zeta-z_j)} \right|\,\mathrm{d}\zeta \\
& \leq  \sum_{k=0}^j \left( |z_{k,\ell}-z_k| \int_\Gamma  \prod_{i_1=k}^{j}\left|\zeta-z_{i_1}\right|^{-1} \prod_{i_2=0}^{k}\left|\zeta-z_{i_2,\ell}\right|^{-1}\,\mathrm{d}\zeta\right).\label{eq:1overprodzeta2}
\end{align}
\end{subequations}
Since the nodes $z_k$ are enclosed by the contour $\Gamma$ we may set $\gamma = \min_{k=0,\ldots,j} |\zeta-z_{k}|>0$ for all $\zeta\in\Gamma$. For a sufficiently large $\ell$ we may assume $|\zeta-z_{k,\ell}|>\gamma/2$ for all $\zeta\in\Gamma$, and consequently, the integral terms in~\eqref{eq:1overprodzeta1} and~\eqref{eq:1overprodzeta2} are bounded from above by $(2/\gamma)^{j+1}$ times the contour length. Combining~\eqref{eq:1overprodzeta2term2} with~\eqref{eq:1overprodzeta2term1} shows~\eqref{eq:convergencegdd}
\end{proof}

{\em Newton-Pad\'e approximation.}
In contrast to polynomial interpolation, solutions to rational interpolation problems may not exist in the classical sense of Hermite interpolation. To avoid difficulties with existence of rational interpolants, we consider rational interpolation in a linearized sense. 
The linearized interpolation problem for given $m$ and $n$ and nodes $z_0,\ldots,z_{m+n}$ consists of finding polynomials $p$ and $q$ of degree $m$ and $n$, respectively, s.t.
\begin{subequations}\label{eq:ratintlin0}
\begin{equation}\label{eq:errorlininterp}
(p-qf)[z_0,\ldots,z_j] = 0,~~~j=0,\ldots,m+n.
\end{equation}
Equivalently to~\eqref{eq:errorlininterp}, the linearized error for $p$ and $q$ has the form
\begin{equation}\label{eq:ratintcondconfluentlin2}
p(z) - q(z) f(z) = \widetilde{v}(z) \prod_{j=0}^{m+n}(z-z_j),
\end{equation}
\end{subequations}
where $\widetilde{v}$ is holomorphic on $E$, cf.~\cite[Section~1]{Cla78}. 
The problem~\eqref{eq:ratintlin0} always has at least one solution $p$ and $q$, since~\eqref{eq:errorlininterp} yields a system of equations with more degrees of freedom than conditions.
Solutions $p$ and $q$ of~\eqref{eq:ratintlin0} may not be unique. However, the quotient $r=p/q\in\mathcal{R}_{mn}$ where $p$ and $q$ solve~\eqref{eq:ratintlin0} is unique
(cf.~\cite{Wu75}), and $p/q$ is also referred to as {\em Newton-Pad\'e approximation} to $f$ for the interpolation nodes $z_0,\ldots,z_{m+n}\in E$.

We let rational interpolants be defined via the linearized interpolation problem in the sequel, and we proceed to use the terms Newton-Pad\'e approximation and rational interpolation equivalently.

\smallskip
{\em Non-degenerate rational functions.}
For given $m$ and $n$, the {\em defect} of $r\in\mathcal{R}_{mn}$ is defined as the maximal number $d\geq 0$ s.t.\ $r\in\mathcal{R}_{m-d,n-d}$. We refer to a rational function as {\em degenerate} if its defect is strictly larger than zero, i.e., $d>0$, and non-degenerate otherwise.

\begin{remark}\label{rmk:zjpoles}
Assume the Newton-Pad\'e approximant $r=p/q\in\mathcal{R}_{mn}$ is non-degenerate. In case of $q(z_j)=0$ for an index $j\in\{0,\ldots,m+n\}$, the identity~\eqref{eq:ratintcondconfluentlin2} implies $p(z_j)=0$ which is contrary to $p$ and $q$ having no common zeros. Thus, in the non-degenerate case the underlying interpolation nodes are distinct to the poles of $r$.
\end{remark}

\smallskip
{\em Rational Hermite interpolation.}
Let $\kappa\leq m+n+1$ denote the number of distinct nodes in $\{z_0,\ldots,z_{m+n}\}$, let $\iota$ denote a mapping s.t.\ the index $\iota(j)$ corresponds to the $j$th distinct node, and let $\eta_j\geq 1$ denote the multiplicity of the node $z_{\iota(j)}$ with $\eta_1+\ldots+\eta_k =m+n+1$. Thus, the underlying interpolation nodes are of the form
\begin{subequations}\label{eq:ratintcondconfluent}
\begin{equation}
\{z_0,\ldots,z_{m+n}\} = \{
\underbrace{z_{\iota(1)},\ldots,z_{\iota(1)}}_{\text{$\eta_1$ many}},
\underbrace{z_{\iota(2)},\ldots,z_{\iota(2)}}_{\text{$\eta_2$ many}},\ldots, \underbrace{z_{\iota(\kappa)},\ldots,z_{\iota(\kappa)}}_{\text{$\eta_\kappa$ many}}\}.
\end{equation}
The Hermite interpolation problem, or osculatory rational interpolation problem, consists of finding $r\in\mathcal{R}_{mn}$ s.t.
\begin{equation}\label{eq:ratintcondconfluent1}
r^{(\ell)}\left( z_{\iota(j)}\right) = 
f^{(\ell)}\left( z_{\iota(j)}\right),
~~~\ell=0,\ldots,\eta_j-1,~~~j=1,\ldots,\kappa.
\end{equation}
\end{subequations}
Contrary to the linearized interpolation problem, the Hermite interpolation problem may have no solution. We proceed to establish a relation between Hermite interpolation and the Newton-Pad\'e approximation.

Assume the Newton-Pad\'e approximant $r=p/q\in\mathcal{R}_{mn}$ is non-degenerate, then the numerator $p$ and denominator $q$ are uniquely defined via solutions of~\eqref{eq:ratintlin0} assuming a certain normalization, e.g., $q(0)=1$. Otherwise, in degenerate cases the numerator $p_0$ and denominator $q_0$ of the Newton-Pad\'e approximant in {\em irreducible} form, i.e., $r=p_0/q_0$ where $p_0$ and $q_0$ have no common zeros, may not solve~\eqref{eq:ratintlin0}.
Following~\cite[Theorem~1]{Wu75} and~\cite{Sa62,Cla78}, there exists a solution $r\in\mathcal{R}_{mn}$ to~\eqref{eq:ratintcondconfluent} if and only if the numerator and denominator of the Newton-Pad\'e approximant in irreducible form satisfy~\eqref{eq:ratintlin0}. In this case, the solution of~\eqref{eq:ratintcondconfluent} coincides with the Newton-Pad\'e approximant provided by~\eqref{eq:ratintlin0}. In particular, this is the case if the Newton-Pad\'e approximant $r\in\mathcal{R}_{mn}$ is non-degenerate.
Assume the non-degenerate case and let $E_0\subset E$ denote a set which contains no poles of $r$, then $r-f$ is holomorphic on $E_0$ and its first divided differences vanish as a consequence of~\eqref{eq:ratintcondconfluent1}, namely,
\begin{equation*}
(r-f)[z_0,\ldots,z_j]=0,~~~j=0,\ldots,m+n.
\end{equation*}
Moreover, the error of the Newton-Pad\'e approximant can be described similar to~\eqref{eq:ratintcondconfluentlin2}. Namely,
\begin{equation}\label{eq:errorinterp2}
r(z)-f(z) = v(z) \prod_{j=0}^{m+n}(z-z_j),~~~z\in E_0,
\end{equation}
where $v$ is holomorphic in $E_0$ and can be derived via $v=\widetilde{v}/q$ with $\widetilde{v}$ as in~\eqref{eq:ratintcondconfluentlin2}.

\begin{remark}\label{rmk:fromzerostonodes}
Assume that for a given $r\in\mathcal{R}_{mn}$ the difference $r-f$ has zeros at points $z_0,\ldots,z_{m+n}\in E$ counting multiplicity, then $r$ satisfies the interpolation conditions~\eqref{eq:ratintcondconfluent}. Consequently, $r$ is a Newton-Pad\'e approximant for the respective nodes.
\end{remark}

\subsection{Pad\'e approximation}
For a survey on Pad\'e approximation we refer to \cite{BG96}. The $(m,n)$-Pad\'e approximation can be understood as Newton-Pad\'e approximation to $f$ with interpolation nodes $z_0=\ldots=z_{m+n}=0$. Considering remarks on rational interpolation, the $(m,n)$-Pad\'e approximation $r^P = p^P/q^P\in\mathcal{R}_{mn}$ always exists in this sense, i.e., $p^P - q^P f = \widetilde{v}_0(z)z^{m+n+1}$ for some remainder $\widetilde{v}_0$ which is holomorphic on $E$. Assume $r^P$ is non-degenerate, then the Pad\'e approximant has no pole at the origin (cf.\ Remark~\ref{rmk:zjpoles}). In particular, there exists a disk $\Delta_0= \rho\Delta$, for a respective radius $\rho>0$, which contains no poles of $r^P$. Consequently, $r^P$ satisfies an error representation of the form~\eqref{eq:errorinterp2}, i.e.,
\begin{equation}\label{eq:Padeasymerror}
\begin{aligned}
r^P(z)-f(z) &=v^0(z)z^{m+n+1},&&z\in \Delta_0,\\
&= a_{mn} z^{m+n+1} + \mathcal{O}(z^{m+n+2}),~~~&&z\to 0,
\end{aligned}
\end{equation}
for some remainder $v^0$ which is holomorphic on $\Delta_0$ and $a_{mn} = v^0(0)$.

Let $c_j$ for $j\geq0$ refer to the coefficients of the Taylor series of $f$ at $z=0$, i.e., $c_j=f^{(j)}(0)/j!$ where $f^{(j)}(0)$ corresponds to the $j$th derivative of $f$ evaluated at zero. Using the convention $c_j=0$ for $j<0$, we use the notation $D_{m,n}^{0}(f)$ for the matrix determinant
\begin{equation*}
D_{m,n}^{0}(f) = \left|
\begin{pmatrix}
c_{m-n+1}& c_{m-n+2} & \cdots & c_m\\
c_{m-n+2}& c_{m-n+3} & \cdots & c_{m+1}\\
\vdots & \vdots & & \vdots \\
c_{m}& c_{m+1} & \cdots &  c_{m+n-1}\\
\end{pmatrix}
\right|.
\end{equation*}
The denominator $q^P$ of the Pad\'e approximation is referred to as $Q^{[m/n]}$ in~\cite{BG96}, and it has the matrix determinant representation \cite[eq.~(1.8)]{BG96}. In particular, 
\begin{equation*}
q^P(0) = D_{m,n}^{0}(f),
\end{equation*}
which is also referred to as Hankel determinant in the literature, cf.~\cite[page~7]{BG96}.
Various notation is used in literature to represent the Hankel determinant $D_{m,n}^{0}(f)$, for instance $(\Delta^{m,n}(f))_{x=0}$ in~\cite{Me67} (as defined on page~169 therein) and $D_{m,n}(f)$ in \cite{GJ76} (choosing the nodes $z_j=0$ therein).

\begin{remark}\label{rmk:DmnfPade}
For the denominator $q^P$ of the Pad\'e approximation, the non-degenerate case implies $q^P(0)\neq 0$ (cf.\ Remark~\ref{rmk:zjpoles}), i.e., $D_{m,n}^{0}(f)\neq 0$.
\end{remark}

\begin{remark}
Similar to~\eqref{eq:Padeasymerror}, we may consider the linearized asymptotic error
\begin{equation}\label{eq:Padeasymerrorlin}
p^P(z) - q^P(z) f(z) = a_{mn} q^P(0) z^{m+n+1} + \mathcal{O}(z^{m+n+2}),~~~z\to 0.
\end{equation}
The error representation~\cite[eq.~(1.11)]{BG96} shows 
\begin{equation*}
p^P(z) - q^P(z) f(z) = -D_{m+1,n+1}^{0}(f) z^{m+n+1} +  \mathcal{O}(z^{m+n+2}),~~~z\to 0.
\end{equation*}
For the non-degenerate case, combining this with~\eqref{eq:Padeasymerrorlin} and $q^P(0) = D_{m,n}^{0}(f)$ we observe
\begin{equation}\label{eq:amnformula}
a_{mn} = -D_{m+1,n+1}^{0}(f) /D_{m,n}^{0}(f).
\end{equation}
\end{remark}

In general, the coefficient $a_{mn}$ can be computed via~\eqref{eq:amnformula}. For the Pad\'e approximation of the exponential function $f=\exp$, an explicit formula is known for the coefficient $a_{mn}$, cf.~\cite[eq.~(10.24)]{Hi08},
\begin{equation}\label{eq:amnexp}
a_{mn} = \frac{(-1)^{n+1} m!n!}{(m+n)!(m+n+1)!},~~~\text{for $f=\exp$}.
\end{equation}

\section{Continuity of rational interpolants}\label{sec:continuity}

Convergence of rational functions in a non-degenerate setting has some relevance in the present work. In particular, sequences of rational approximants $\{r^\varepsilon\}_{\varepsilon>0}$, where $r^\varepsilon \in\mathcal{R}_{mn}$, which converge to a non-degenerate rational function $r\in\mathcal{R}_{mn}$. In the present work, the most relevant types of convergence are
\begin{enumerate}
\item[(au)] ``almost uniform'' convergence, i.e., $\| r^\varepsilon - r \|_{K'} \to 0$ for every compact set $K'\subset \mathbb{C}$ which contains no poles of $r$, and
\item[(cw)] ``coefficient-wise'' convergence, i.e., the coefficients of the numerator and denominator of $r^\varepsilon$ converge to the respective coefficients of $r$.
\end{enumerate}
Under the assumption that $r$ is non-degenerate these types of convergence are equivalent, cf.~\cite[Theorem~1]{TG85} or~\cite[Theorem~1]{Gu90}. Since we only consider non-degenerate cases in the present work, the notation $r^\varepsilon\to r$ refers to convergence almost uniform and coefficient-wise in an equivalent manner.

Partly based on the convergence result in~\cite[Theorem~9]{GJ76}, we show the following auxiliary results.
\begin{proposition}\label{prop:cwconvergence}
Assume the $(m,n)$-Pad\'e approximant $r^P$ to $f$ is non-degenerate.
Let $\{(z_0(\varepsilon),\ldots,z_{m+n}(\varepsilon))\}_{\varepsilon>0}$ denote a sequence of nodes with $z_j(\varepsilon)\in E$, and let $\{r^\varepsilon\}_{\varepsilon>0}$ denote the respective sequence of Newton-Pad\'e approximants, i.e., $r^\varepsilon\in\mathcal{R}_{mn}$ is a rational interpolant to $f$ with interpolation nodes $z_0(\varepsilon),\ldots,z_{m+n}(\varepsilon)$. Assume that $z_j(\varepsilon)\to 0$ for $j=0,\ldots,m+n$ and $\varepsilon \to 0$, then the following statements hold true.
\begin{enumerate}[label=(\roman*)]
\item\label{item:convreps} The interpolant $r^\varepsilon$ converges to the Pad\'e approximant, i.e., $r^\varepsilon\to r^P$ for $\varepsilon\to 0$.
\item\label{item:repsnondeg} For a sufficiently small $\varepsilon$, the interpolant $r^{\varepsilon}$ is non-degenerate.
\item\label{item:repspoles} If $\Delta_0\subset E$ denotes a disk which contains no poles of $r^P$, then the interpolant $r^\varepsilon$ has no poles on $\Delta_0$ for a sufficiently small $\varepsilon$.
\item\label{item:repserror} For a fixed disk $\Delta_0$ as above and sufficiently small $\varepsilon$, the interpolant $r^\varepsilon$ satisfies
\begin{equation}\label{eq:asymerrorinterp0}
r^\varepsilon(z)-f(z) = v^\varepsilon(z) \prod_{j=0}^{m+n}(z-z_j(\varepsilon)),~~~z\in \Delta_0,
\end{equation}
for a remainder $v^\varepsilon$ which is holomorphic on $\Delta_0$.
\item\label{item:amncont0}
The remainder $v^\varepsilon$ satisfies
\begin{equation}\label{eq:vepsOe0}
v^\varepsilon(\varepsilon z)
= a_{mn} + \sum_{j=0}^{m+n} \mathcal{O}\left(|z_j(\varepsilon)|\right) + 
\mathcal{O}(\varepsilon |z|),~~~\text{for $\varepsilon\to 0$,}
\end{equation}
with $a_{mn}=v^0(0)$ as in~\eqref{eq:Padeasymerror}.
\end{enumerate}
\end{proposition}

\begin{proof} We first prove~\ref{item:convreps} by applying~\cite[Theorem~9]{GJ76}. This requires the condition in~\cite[Theorem~9.(B)]{GJ76} and $D_{m,n}(f)\neq 0 $ in~\cite[eq.~(4.16)]{GJ76} to hold true. The rational interpolant for the nodes $z_j=0$ for $j=0,\ldots,m+n$ corresponds to the Pad\'e approximant $r^P$. The term $D_{m,n}(f)$ for $z_j=0$ refers to the underlying Hankel determinant, and following Remark~\ref{rmk:DmnfPade} we have $D_{m,n}(f)\neq 0$ in the non-degenerate case. Moreover,~\cite[Theorem~9.(B)]{GJ76} requires $r^P$ to solve the linearized interpolation problem in irreducible form which holds true for the non-degenerate case as well. Thus, we may apply~\cite[Theorem~9]{GJ76} to show point-wise convergence of $r^\varepsilon$ to $r^P$ on any compact sub-set of $\mathbb{C}$ which contains no poles of $r^P$. Consequently, $r^\varepsilon\to r^P$ ``almost uniform'' when the underlying interpolation nodes converge to the origin $z_j(\varepsilon)\to 0$ which proves the assertion~\ref{item:convreps}.

We proceed with the proof of~\ref{item:repsnondeg}. Due to the assumption that $r^P$ is non-degenerate, i.e., $r^P\not\in \mathcal{R}_{m-d,n-d}$ for $d>0$, the Pad\'e approximant has either exactly $m$ zeros or exactly $n$ poles in $\mathbb{C}$, and its zeros and poles are distinct. This carries over from $r^P$ to $r^\varepsilon$ as following. Assume $r^P$ has exactly $m$ zeros, then coefficient-wise convergence $r^\varepsilon\to r^P$ implies that the zeros of $r^\varepsilon$ converge to the zeros of $r^P$ (up to ordering, cf.~\cite[Proposition~5.2.1]{Ar11}), and the poles of $r^\varepsilon$ either converge to the poles of $r^P$ or to $\infty$, where the latter occurs if $r^P$ has $<n$ poles. Thus, for sufficiently small $\varepsilon$, the rational function $r^\varepsilon$ has exactly $m$ zeros and its zeros and poles are distinct, i.e. $r^\varepsilon\not\in \mathcal{R}_{m-d,n-d}$ for $d>0$. Similar arguments hold true for the case that $r^P$ has exactly $n$ poles, which proves~\ref{item:repsnondeg}.

Proof of~\ref{item:repspoles}. Coefficient-wise convergence $r^\varepsilon\to r^P$ implies that the poles of $r^\varepsilon$ either converge to the poles of $r^P$ or to $\infty$, and since the disk $\Delta_0$ contains no poles of $r^P$ and is compact, this proves the assertion of~\ref{item:repspoles}.

Proof of~\ref{item:repserror}. Let $\Delta_0\subset E$ denote a disk set s.t.\ $r^P$ has no poles on $\Delta_0$. Following~\ref{item:repsnondeg} and~\ref{item:repspoles}, the interpolant $r^\varepsilon$ is non-degenerate and has no poles on $\Delta_0$ for a sufficiently small $\varepsilon$. Thus, in this case arguments around~\eqref{eq:errorinterp2} imply that a representation of the form~\eqref{eq:asymerrorinterp0} holds true.

Proof of~\ref{item:amncont0}. Similar to~\eqref{eq:ratintcondconfluentlin2}, the rational interpolant $r^\varepsilon = p^\varepsilon/q^\varepsilon$ with remainder $\widetilde{v}^\varepsilon$ satisfies the linearized problem
\begin{equation}\label{eq:vepslinerr}
p^\varepsilon(z) - q^\varepsilon(z) f(z) = \widetilde{v}^\varepsilon(z) \prod_{j=0}^{m+n}(z-z_j(\varepsilon)).
\end{equation}
Explicit representations of $p^\varepsilon$, $q^\varepsilon$ and $\widetilde{v}^\varepsilon$ subject to~\eqref{eq:vepslinerr} are provided in~\cite[Theorem~7.1.1]{BG96}.
Namely, define
\begin{equation*}
f^\varepsilon_{\ell,k} = f[z_\ell(\varepsilon),\ldots,z_k(\varepsilon)],~~~\ell\leq k,
\end{equation*}
and $f^\varepsilon_{\ell,k} =0$ for $k<\ell$, and define the matrices
\begin{subequations}\label{eq:defHEqv}
\begin{equation}
\begin{aligned}
H^\varepsilon(z) = & \begin{pmatrix}
f^\varepsilon_{n,m+1}& f^\varepsilon_{n,m+2} & \cdots & f^\varepsilon_{n,m+n} & 
\prod_{k=0}^{n-1}(z-z_k(\varepsilon))\\
f^\varepsilon_{n-1,m+1}& f^\varepsilon_{n-1,m+2} & \cdots & f^\varepsilon_{n-1,m+n} &
\prod_{k=0}^{n-2}(z-z_k(\varepsilon))\\
\vdots & \vdots & & \vdots & \vdots \\
f^\varepsilon_{0,m+1}& f^\varepsilon_{0,m+2} & \cdots &  f^\varepsilon_{0,m+n} & 
1\\
\end{pmatrix}\\
&\in\mathbb{C}^{n+1\times n+1},
\end{aligned}
\end{equation}
and
\begin{equation}
\begin{aligned}
E^\varepsilon(z) = &
\begin{pmatrix}
f^\varepsilon_{n,m+1}& f^\varepsilon_{n,m+2} & \cdots & f^\varepsilon_{n,m+n} & f[z_n(\varepsilon),\ldots,z_{m+n}(\varepsilon),z]\\
f^\varepsilon_{n-1,m+1}& f^\varepsilon_{n-1,m+2} & \cdots & f^\varepsilon_{n-1,m+n} & f[z_{n-1}(\varepsilon),\ldots,z_{m+n}(\varepsilon),z]\\
\vdots & \vdots & & \vdots & \vdots \\
f^\varepsilon_{0,m+1}& f^\varepsilon_{0,m+2} & \cdots &  f^\varepsilon_{0,m+n} & f[z_0(\varepsilon),\ldots,z_{m+n}(\varepsilon),z]\\
\end{pmatrix}\\
&\in\mathbb{C}^{n+1\times n+1}.
\end{aligned}
\end{equation}
Then, following~\cite[Theorem~7.1.1]{BG96}, the denominator $q^\varepsilon$ and the remainder $\widetilde{v}^\varepsilon$ correspond to the matrix determinants
\begin{equation}
q^\varepsilon(\varepsilon z) = \det H^\varepsilon(\varepsilon z)
~~~\text{and}~~
\widetilde{v}^\varepsilon(\varepsilon z) = - \det E^\varepsilon(\varepsilon z).
\end{equation}
\end{subequations}
For a given $z\in\mathbb{C}$ and sufficiently small $\varepsilon$ s.t.\ $\varepsilon z\in\Delta_0$ and $q(\varepsilon z)\neq 0$, this yields a representation of the remainder $v^\varepsilon$ in~\eqref{eq:asymerrorinterp0}, i.e.,
\begin{equation*}
v^\varepsilon(\varepsilon z) = \widetilde{v}^\varepsilon(\varepsilon z)/q^\varepsilon(\varepsilon z).
\end{equation*}
This representation remains true for the Pad\'e approximation with $\varepsilon=0$, i.e., $q^0=q^P$ and $v^0 = \widetilde{v}^0/q^0$ for $v^0$ as in~\eqref{eq:Padeasymerror}. In particular, 
\begin{equation*}
a_{mn} = v^0(0)=\widetilde{v}^0(0)/q^0(0).
\end{equation*}
To prove~\eqref{eq:vepsOe0}, we consider the difference
\begin{align}
v^\varepsilon(\varepsilon z)-a_{mn}
& = \widetilde{v}^\varepsilon(\varepsilon z)/q^\varepsilon(\varepsilon z)
- \widetilde{v}^0(0)/q^0(0)\notag\\
& = \frac{\left(\widetilde{v}^\varepsilon(\varepsilon z)-\widetilde{v}^0(0)\right) q^0(0) - \widetilde{v}^0(0)\left(q^\varepsilon(\varepsilon z)-q^0(0)\right)}{q^0(0)q^\varepsilon(\varepsilon z)}. \label{eq:vepsmamnfrac}
\end{align}
We proceed by showing
\begin{subequations}\label{eq:vepsminsamnOnosubs}
\begin{equation}\label{eq:vepsminsamnOnov}
q^\varepsilon(\varepsilon z)-q^0(0) = \sum_{j=0}^{m+n} \mathcal{O}\left(|z_j(\varepsilon)|\right) + \mathcal{O}(\varepsilon|z|) ,~~~\text{for $\varepsilon\to 0$},
\end{equation}
and
\begin{equation}\label{eq:vepsminsamnOnoq}
\widetilde{v}^\varepsilon(\varepsilon z)-\widetilde{v}^0(0)
= \sum_{j=0}^{m+n} \mathcal{O}\left(|z_j(\varepsilon)|\right) + \mathcal{O}(\varepsilon|z|) ,~~~\text{for $\varepsilon\to 0$},
\end{equation}
\end{subequations}
by using the representation~\eqref{eq:defHEqv} and properties of matrix determinants.

We first show~\eqref{eq:vepsminsamnOnov}. Following \cite[Corollary~2.7]{IR08}, the difference of $ q^\varepsilon(\varepsilon z)=\det H^\varepsilon(\varepsilon z) $ and $q^0(0)=\det H^0(0)$ satisfies
\begin{equation}\label{eq:detHdiff}
|\det H^\varepsilon(\varepsilon z)-\det H^0(0)|\leq \sum_{j=1}^{n+1}
s_{n+1-j} \|H^\varepsilon(\varepsilon z)-H^0(0)\|_2^j,
\end{equation}
where $s_1,\ldots,s_{n+1}$ are elementary symmetric functions in the singular values of $H^0(0)$ as defined in this reference, and $\|\cdot\|_2$ therein refers to the Euclidean matrix norm.
Moreover, the Euclidean matrix norm therein is bounded by, cf.\ \cite[Table~6.2]{Hi02},
\begin{equation}\label{eq:detH2normest}
\|H^\varepsilon(\varepsilon z)-H^0(0)\|_2 \leq \sum_{\ell=1}^{n+1}\sum_{k=1}^{n+1} |(H^\varepsilon(\varepsilon z))_{\ell k}-(H^0(0))_{\ell k}|.
\end{equation}
We proceed to derive asymptotic estimates for the difference of the matrix entries therein.
For $k=1,\ldots,n$ and $\ell =1,\ldots,n+1$, i.e., the first $n$ columns, Proposition~\ref{prop:ddcont} implies
\begin{equation}\label{eq:Hlkfirst}
\begin{aligned}
(H^\varepsilon(\varepsilon z))_{\ell k}-(H^0(0))_{\ell k}
&= f[z_{n-\ell+1}(\varepsilon),\ldots,z_{m+k}(\varepsilon)]
- f[0,\ldots,0]\\
&=\sum_{j=0}^{m+n} \mathcal{O}\left(|z_j(\varepsilon)|\right)  ,~~~\text{for $\varepsilon\to 0$}.
\end{aligned}
\end{equation}
The matrix entries in the last column of $H$, i.e., $k=n+1$, satisfy
\begin{subequations}\label{eq:Hlastcolumnest}
\begin{equation}\label{eq:Hlastcolumn}
(H^\varepsilon(\varepsilon z))_{\ell k}-(H^0(0))_{\ell k}
=\left\{\begin{array}{ll}
\prod_{j=0}^{n-\ell+1}(\varepsilon z-z_j(\varepsilon)),~&\ell=1,\ldots,n,~\text{and}\\
0,&\ell=n+1.
\end{array}\right.
\end{equation}
For the case $\ell=1,\ldots,n$, the difference in~\eqref{eq:Hlastcolumn} is bounded by
\begin{equation}
\begin{aligned}
\prod_{j=0}^{n-\ell+1}|\varepsilon z-z_j(\varepsilon)|
&\leq 2^{n-\ell+2}\left((\varepsilon |z|)^{n-\ell+2} + \max_{i=0,\ldots,j}|z_i(\varepsilon)|^{n-\ell+2}\right)\\
&= \sum_{j=0}^{m+n} \mathcal{O}\left(|z_j(\varepsilon)|\right) + \mathcal{O}(\varepsilon|z|) ,~~~\text{for $\varepsilon\to 0$}.
\end{aligned}
\end{equation}
\end{subequations}
Combining~\eqref{eq:detH2normest},~\eqref{eq:Hlkfirst} and~\eqref{eq:Hlastcolumnest}, we arrive at
\begin{equation*}
\|H^\varepsilon(\varepsilon z)-H^0(0)\|_2 = \sum_{j=0}^{m+n} \mathcal{O}\left(|z_j(\varepsilon)|\right) + \mathcal{O}(\varepsilon|z|),~~~\text{for $\varepsilon\to 0$}.
\end{equation*}
Since higher powers of the Euclidean matrix norm in~\eqref{eq:detHdiff} satisfy the same asymptotic bound, this shows~\eqref{eq:vepsminsamnOnov}. 

We proceed to show~\eqref{eq:vepsminsamnOnoq} analogously. 
The difference $(E^\varepsilon(\varepsilon z))_{\ell k}-(E^0(0))_{\ell k}$ satisfies~\eqref{eq:Hlkfirst} for $k=1,\ldots,n$ and $\ell =1,\ldots,n+1$ since the first $n$ columns of $E^\varepsilon$ and $H^\varepsilon$ are identical. Following Proposition~\ref{prop:ddcont}, entries of the last column, i.e., $k=n+1$, satisfy
\begin{equation*}
\begin{aligned}
(E^\varepsilon(\varepsilon z))_{\ell k}-(E^0(0))_{\ell k}
&= f[z_{n-\ell+1}(\varepsilon),\ldots,z_{m+k}(\varepsilon),\varepsilon z]
- f[0,\ldots,0]\\
&= \sum_{j=0}^{m+n} \mathcal{O}\left(|z_j(\varepsilon)|\right) + \mathcal{O}(\varepsilon|z|) ,~~~\text{for $\varepsilon\to 0$}.
\end{aligned}
\end{equation*}
Using estimates similar to~\eqref{eq:detHdiff} and~\eqref{eq:detH2normest} for the difference of $\widetilde{v}^\varepsilon(\varepsilon z) = -\det E^\varepsilon(\varepsilon z)$ and $\widetilde{v}^0(0) = -\det E^0(0)$, we conclude~\eqref{eq:vepsminsamnOnoq}.
Combining~\eqref{eq:vepsmamnfrac} with~\eqref{eq:vepsminsamnOnosubs}, where the denominator in~\eqref{eq:vepsmamnfrac} can be further bounded, we conclude~\eqref{eq:vepsOe0}.
\end{proof}

The following auxiliary result holds true when considering convergence to the Pad\'e approximant.
\begin{proposition}\label{prop:zerosreps}
Assume the $(m,n$)-Pad\'e approximant $r^P$ to $f$ is non-degenerate with $a_{mn}\neq 0$, and let $\{r^\varepsilon\}_{\varepsilon>0}$ denote a sequence of $(m,n)$-rational functions with $r^\varepsilon\to r^P$ as $\varepsilon\to 0$. Then there exists $\alpha>0$ s.t.\ for sufficiently small $\varepsilon$, the disk $\alpha\Delta$ contains exactly $m+n+1$ zeros of $r^\varepsilon-f$ counting multiplicity. Moreover, the zeros of $r^\varepsilon-f$ converge to the origin for $\varepsilon\to 0$.
\end{proposition}
\begin{proof}
By construction, the error of the non-degenerate Pad\'e approximation has a zero at $z=0$ of multiplicity $m+n+1$. We proceed to show that on a sufficiently small disk $\alpha \Delta$, the error $r^P-f$ attains no further zeros. For a sufficiently small $\alpha>0$, the error curve $r^P(z)-f(z)=a_{mn}z^{m+n+1}+\mathcal{O}(z^{m+n+2})$ for $|z|=\alpha$ has a winding number of $m+n+1$. Let $\Delta_0\subset E$ denote a disk which contains no poles of $r^P$, and assume $\alpha\Delta\subset \Delta_0$ which certainly holds true for sufficiently small $\alpha$. Consequently, $r^P-f$ has no poles on $\alpha \Delta$. Following the argument principle, the number of zeros of $r^P-f$ in $\alpha\Delta$ corresponds to the winding number of $r^P(z)-f(z)$ for $|z|=\alpha$, i.e., $r^P(z)-f(z)$ has exactly $m+n+1$ zeros counting multiplicity in $\alpha \Delta$.

Let $\alpha$ be fixed as above. Define $\beta:=\min_{|z|=\alpha} |r^P(z)-f(z)|>0$ which is strictly positive since the only zeros of $r^P-f$ on $\alpha\Delta$ are at $z=0$. Since $\alpha\Delta\subset \Delta_0$ the convergence $r^\varepsilon\to r^P$ can be understood as uniform convergence on the disk $\alpha\Delta$.
For a sufficiently small $\varepsilon$, we have $|r^\varepsilon(z) - r^P(z)| < \beta/2$ for $|z|=\alpha$, and thus, $r^\varepsilon(z)-f(z)$ also has a winding number of $m+n+1$ for $|z|=\alpha$. Moreover, for a sufficiently small $\varepsilon$, Proposition~\ref{prop:cwconvergence}.\ref{item:repspoles} shows that $r^\varepsilon$ has no poles on $\alpha\Delta\subset \Delta_0$, and we may use the argument principle to show that $r^\varepsilon-f$ has exactly $m+n+1$ zeros in $\alpha\Delta$ counting multiplicity.

Let $z_0(\varepsilon),\ldots,z_{m+n}(\varepsilon)$ correspond to the zeros of $r^\varepsilon-f$ on $\alpha\Delta$.
We may choose $\alpha'\leq \alpha$ arbitrary small and repeat the arguments above for the disk $\alpha'\Delta$ to show that the zeros of $r^\varepsilon-f$ converge to the origin for $\varepsilon\to 0$, i.e., $z_j(\varepsilon)\to 0$ for $\varepsilon \to 0$ and $j=0,\ldots,m+n$.

\end{proof}

\section{Error of rational interpolants with scaled nodes}\label{sec:asymerrinterp}

\begin{proposition}\label{prop:convinterp}
Assume the $(m,n$)-Pad\'e approximant to $f$ is non-degenerate. Let $\zeta_0,\ldots,\zeta_{m+n}\in\Theta$ be given nodes for a compact set $\Theta\subset\mathbb{C}$, and let $r^\varepsilon\in\mathcal{R}_{mn}$ denote the Newton-Pad\'e approximant to $f$ for the interpolation nodes $z_0(\varepsilon),\ldots,z_{m+n}(\varepsilon)$ with $z_j(\varepsilon)= \varepsilon \zeta_{j}$ for $j=0,\ldots,m+n$. Then, there exists an upper bound $L_1>0$ s.t.
\begin{subequations}\label{eq:errorlimitinterpboth}
\begin{equation}\label{eq:errorlimitinterp0}
\limsup_{\varepsilon\to 0} \frac{\left|r^\varepsilon(\varepsilon z)-f(\varepsilon z) - a_{mn} \varepsilon^{m+n+1} \prod_{j=0}^{m+n}(z-\zeta_j)\right|}{\varepsilon^{m+n+2}} \leq L_1 < \infty,
\end{equation}
for all $z\in K$ and $\zeta_0,\ldots,\zeta_{m+n}\in\Theta$, with $a_{mn}$ as in~\eqref{eq:Padeasymerror}.
In other words, using $\mathcal{O}$-notation we have
\begin{equation}
r^\varepsilon(\varepsilon z)-f(\varepsilon z)
= a_{mn} \varepsilon^{m+n+1} \prod_{j=0}^{m+n}(z-\zeta_j)  + \mathcal{O}(\varepsilon^{m+n+2}),
~~~\varepsilon\to 0,
\end{equation}
\end{subequations}
uniformly for $z\in K$ and $\zeta_0,\ldots,\zeta_{m+n}\in\Theta$.

In a similar manner, there exists a constant $L_2>0$ s.t.\ the uniform error satisfies
\begin{subequations}\label{eq:errorlimitinterpuniformboth}
\begin{equation}\label{eq:errorlimitinterpuniform0}
\limsup_{\varepsilon\to 0} \frac{\left|\|r^\varepsilon-f\|_{\varepsilon K} - |a_{mn}| \varepsilon^{m+n+1} \max_{z\in K} \prod_{j=0}^{m+n} |z-\zeta_j|\right|}{\varepsilon^{m+n+2}} \leq L_2 < \infty,
\end{equation}
for all $\zeta_0,\ldots,\zeta_{m+n}\in\Theta$.
Using $\mathcal{O}$-notation we have
\begin{equation}\label{eq:errorlimitinterpuniform}
\|r^\varepsilon-f\|_{\varepsilon K}
= |a_{mn}| \varepsilon^{m+n+1} \max_{z\in K} \prod_{j=0}^{m+n} |z-\zeta_j|  + \mathcal{O}(\varepsilon^{m+n+2}),
~~~\varepsilon\to 0,
\end{equation}
\end{subequations}
uniformly for $\zeta_0,\ldots,\zeta_{m+n}\in\Theta$.
\end{proposition}
\begin{proof}
Since the underlying interpolation nodes converge to the origin, i.e., $z_j(\varepsilon)\to 0$ for $j=0,\ldots,m+n$, Proposition~\ref{prop:cwconvergence} implies that for a sufficiently small $\varepsilon$ the interpolant satisfies Hermite interpolation conditions~\eqref{eq:asymerrorinterp0} for a remainder $v^\varepsilon$ depending on the choices of the underlying nodes $\zeta_j$. Substituting $\varepsilon z$ for $z$ in~\eqref{eq:asymerrorinterp0} under the assumption that $\varepsilon z\in \Delta_0$ for $\Delta_0$ as in Proposition~\ref{prop:cwconvergence}.\ref{item:repspoles}, we arrive at
\begin{equation}\label{eq:fromrmftovepseps}
r^\varepsilon(\varepsilon z)-f(\varepsilon z) = v^\varepsilon(\varepsilon z) \varepsilon^{m+n+1} \prod_{j=0}^{m+n}(z- \zeta_j),~~~\varepsilon z\in \Delta_0.
\end{equation}
This implies
\begin{equation}\label{eq:asymerrorinterp}
\frac{\left|r^\varepsilon(\varepsilon z)-f(\varepsilon z)
- a_{mn} \varepsilon^{m+n+1} \prod_{j=0}^{m+n}(z-\zeta_j)\right|}{\varepsilon^{m+n+2}}
= \frac{\left|v^\varepsilon(\varepsilon z) - a_{mn}\right|}{\varepsilon}  \prod_{j=0}^{m+n}|z-\zeta_j|.
\end{equation}

We proceed to apply Proposition~\ref{prop:cwconvergence}.\ref{item:amncont0}.
Since $z_j(\varepsilon)=\varepsilon \zeta_j$ for $\zeta_j\in\Theta$, $j=0,\ldots,m+n$, where $\Theta$ is compact, and $z\in K$ where $K$ is compact, the convergence result in~\eqref{eq:vepsOe0} implies \begin{equation}\label{eq:vepsamnbounded}
\max_{z\in K}|v^\varepsilon(\varepsilon z) - a_{mn}| \leq \kappa \varepsilon,
\end{equation}
for a constant $\kappa>0$ and sufficiently small $\varepsilon$ independent of the explicit choice of $\zeta_0,\ldots,\zeta_{m+n}\in\Theta$. The upper bound~\eqref{eq:vepsamnbounded} together with compactness arguments yield that the right-hand side in~\eqref{eq:asymerrorinterp} is bounded uniformly for $\zeta_0,\ldots,\zeta_{m+n}\in\Theta$, and substituting~\eqref{eq:asymerrorinterp} in~\eqref{eq:errorlimitinterp0}, we conclude that~\eqref{eq:errorlimitinterpboth} holds true.

We proceed to show~\eqref{eq:errorlimitinterpuniformboth}.
Making use of a reverse triangular inequality and~\eqref{eq:fromrmftovepseps}, we observe
\begin{align}
&
\frac{\left|\|r^\varepsilon-f\|_{\varepsilon K} - |a_{mn}| \varepsilon^{m+n+1} \max_{z\in K} \prod_{j=0}^{m+n} |z-\zeta_j|\right|}{\varepsilon^{m+n+2}}
\notag\\
&\qquad\leq 
\frac{\max_{z\in K}\left|r^\varepsilon(\varepsilon z)-f(\varepsilon z)
- a_{mn} \varepsilon^{m+n+1} \prod_{j=0}^{m+n}(z-\zeta_j)\right|}{\varepsilon^{m+n+2}}
\notag\\
&\qquad \leq \frac{\max_{z\in K}   |v^\varepsilon(\varepsilon z) - a_{mn}|}{\varepsilon} \max_{z\in K} \prod_{j=0}^{m+n} |z-\zeta_j|.\label{eq:rmfconvdetailmaxbound}
\end{align}
Similar to~\eqref{eq:asymerrorinterp}, the right-hand side in~\eqref{eq:rmfconvdetailmaxbound} is bounded uniformly for $\zeta_0,\ldots,\zeta_{m+n}\in\Theta$ due to~\eqref{eq:vepsamnbounded} and compactness arguments. Substituting~\eqref{eq:asymerrorinterp} in~\eqref{eq:errorlimitinterpuniform0} we conclude that~\eqref{eq:errorlimitinterpuniformboth} holds true.
\end{proof}

\section{Chebyshev nodes}\label{sec:chebyshevnodes}

Chebyshev polynomials of degree $m+n+1$ have some relevance for $(m,n)$-rational approximation.
Since we assume that $K$ has no isolated points, it has infinitely many points, and consequently, $K$ is a Chebyshev set for polynomials in $\mathbb{C}$, i.e., polynomials on $K$ satisfy the Haar condition, cf.~\cite[Theorem~19]{Me67}.
This implies that polynomial Chebyshev approximants on $K$ exist and are unique. In particular there exists a polynomial $g$ of degree $m+n$ which uniquely minimizes $\|z^{m+n+1} - g\|_K$.
The zeros of the polynomial $z^{m+n+1} - g$ are also referred to as the $m+n+1$ Chebyshev nodes of $K$.
Let $\tau_0,\ldots,\tau_{m+n}\in\mathbb{C}$ denote the Chebyshev nodes, then for any sequence $\zeta_0,\ldots,\zeta_{m+n}\in\mathbb{C}$ distinct to the sequence of Chebyshev nodes, 
\begin{equation}\label{eq:chebconst}
\max_{z\in K} \prod_{j=0}^{m+n} |z-\zeta_j|
> \max_{z\in K} \prod_{j=0}^{m+n} |z-\tau_j|
~~~ =: t_{m+n+1}.
\end{equation}
The polynomial $z^{m+n+1}-g(z)$, potentially re-scaled, is also referred to as Chebyshev polynomial of degree $m+n+1$.
We remark that the Chebyshev nodes are located in the convex hull of $K$.

{\em An example, the interval $K=[-1,1]$.}
Following classical results, for instance \cite[Corollary~8.1]{SM03}, the Chebyshev nodes for the interval correspond to
\begin{equation}\label{eq:ratinterpolateCheb}
\begin{aligned}
&\tau_{j} = \cos\left(\frac{(2j+1)\pi}{2(m+n+1)}\right),~~~j=0,\ldots,m+n,~~~\text{and}\\
&t_{m+n+1}=2^{-(m+n)},~~~\text{ for the interval $K=[-1,1]$.}
\end{aligned}
\end{equation}

{\em Another example, the unit disk $K=\Delta$.}
On the unit disk the Chebyshev nodes correspond to
\begin{equation*}
\begin{aligned}
&\tau_j = 0,~~~j=0,\ldots,m+n,~~~\text{and}\\
&t_{m+n+1}=1,~~~\text{ for the unit disk $K=\Delta$,}
\end{aligned}
\end{equation*}
cf.~\cite[Example~5]{LS06b}.

\subsection*{Rational interpolation at Chebyshev nodes}

The term $\max_{z\in K} \prod_{j=0}^{m+n} |z-\zeta_j|$ appears in~\eqref{eq:errorlimitinterpuniformboth} as a factor for the asymptotic uniform error of the rational interpolant for the interpolation nodes $z_0(\varepsilon),\ldots, z_{m+n}(\varepsilon)$ where $z_j(\varepsilon)=\varepsilon\zeta_j$, $j=0,\ldots,m+n$. 
Following remarks of the present section, particularly~\eqref{eq:chebconst}, this factor is minimized when choosing the underlying nodes $\zeta_0,\ldots,\zeta_{m+n}$ to be the $m+n+1$ Chebyshev nodes $\tau_0,\ldots,\tau_{m+n}$ of $K$. More precisely, let $r^\varepsilon_\tau\in\mathcal{R}_{mn}$ denote the rational interpolant for $z_j(\varepsilon)=\varepsilon \tau_j$, $j=0,\ldots,m+n$. Then substituting $\tau_j$ for $\zeta_j$ in~\eqref{eq:errorlimitinterpuniformboth} in Proposition~\ref{prop:convinterp} and making use of~\eqref{eq:chebconst}, we observe
\begin{equation}\label{eq:errorlimitinterpchebnodes}
\|r^\varepsilon_\tau(\varepsilon z)-f(\varepsilon z)\|_{\varepsilon K}
= |a_{mn}| t_{m+n+1} \varepsilon^{m+n+1}   + \mathcal{O}(\varepsilon^{m+n+2}),
~~~\varepsilon\to 0.
\end{equation}
When considering rational interpolants with scaled interpolation nodes, the Chebyshev nodes yield an optimal asymptotically leading order term for the uniform error~\eqref{eq:errorlimitinterpuniform}.

\section{Interpolatory best approximations}\label{sec:ibest}

Let $\mathcal{I}_{mn}^{\Theta}\subset\mathcal{R}_{mn}$ for $\Theta\subset E$ refer to the set of rational functions in $\mathcal{R}_{mn}$ which interpolate $f$ in the Newton-Pad\'e sense at $m+n+1$ nodes $z_0, \ldots, z_{m+n} \in\Theta$ counting multiplicity.
In particular,
\begin{equation*}
\mathcal{I}_{mn}^{\Theta} := \{r\in\mathcal{R}_{mn}~|~\text{$r$ interpolates $f$ at nodes $z_0,\ldots,z_{m+n}\in\Theta$}\}.
\end{equation*}
For any sequence of $m+n+1$ interpolation nodes in $\Theta$, there exists a unique Newton-Pad\'e approximant in $\mathcal{R}_{mn}$ to $f$, and thus, the set $\mathcal{I}_{mn}^{\Theta}$ is non-empty provided $\Theta$ is non-empty which we assume in the sequel.

\begin{proposition}[Existence of interpolatory best approximants]\label{prop:bestinterp}
Let $m$ and $n$ be given degrees and let $\Theta\subset E$ be a given compact set. Assume there exists $u\in\mathcal{I}_{mn}^{\Theta}$ with $\|u-f\|_K<\infty$. Then there exists $\widehat{r}\in\mathcal{I}_{mn}^{\Theta}$ which minimizes the uniform error on $K\subset E$ in the set $\mathcal{I}_{mn}^{\Theta}$, i.e.,
\begin{equation}\label{eq:infinITK}
\|\widehat{r} - f\|_{K}
= \inf_{r\in\mathcal{I}_{mn}^{\Theta}} \|r - f\|_{K}<\infty.
\end{equation}
We may also refer to this infimum as a minimum in the sequel.
\end{proposition}
\begin{proof}
It may occur that an interpolant $r\in\mathcal{I}_{mn}^{\Theta}$ has poles in $K$, in which case $ \|r - f \|_{K} = \infty$.
Since we assume there exists at least one rational interpolant $u\in\mathcal{I}_{mn}^{\Theta}$ with $\|u-f\|_K<\infty$, there exists a minimizing sequence $\{\widetilde{r}_k\}_{k\in\mathbb{N}}$, $\widetilde{r}_k\in\mathcal{I}_{mn}^{\Theta}$, for the uniform error on $K$, i.e.,
\begin{equation*}
\lim_{k\to \infty} \|\widetilde{r}_k - f \|_{K}
= \inf_{r\in\mathcal{I}_{mn}^{\Theta}} \|r - f\|_{K}<\infty.
\end{equation*}
Moreover, we may assume that $\|\widetilde{r}_k\|_K$ is bounded for $k\in\mathbb{N}$, in particular, $\|\widetilde{r}_k\|_K\leq \|u-f\|_K +\|f\|_K<\infty$.
Let $\widetilde{r}_k = \widetilde{p}_k/\widetilde{q}_k$ where $\widetilde{p}_k$ and $\widetilde{q}_k$ are polynomials of degree $\leq m$ and $\leq n$, respectively.
Similar to~\cite{Wa31} or the proof of~\cite[Theorem~24.1]{Tre13} we may normalize the quotient $\widetilde{p}_k/\widetilde{q}_k$ by setting $\|\widetilde{q}_k\|_K=1$. This further implies $\|\widetilde{p}_k\|_K \leq \|\widetilde{r}_k\|_K$, and thus, $\widetilde{q}_k$ and $\widetilde{p}_k$ are contained in compact sets.

Consequently, there exists a sub-sequence $k_\ell$ s.t.\ $\widetilde{p}_{k_\ell}\to \widetilde{p}$ and $\widetilde{q}_{k_\ell}\to \widetilde{q}$ coefficient-wise where $\widetilde{p}$ and $\widetilde{q}$ are polynomials of degree $\leq m$ and $\leq n$, respectively. Define $\widetilde{r} = \widetilde{p}/\widetilde{q}\in\mathcal{R}_{mn}$. For any point $z\in K$ which is not a zero of $\widetilde{q}$, we observe
\begin{equation}\label{eq:uerrornopoleproof}
|\widetilde{r}(z)-f(z)| = \lim_{\ell\to\infty}| \widetilde{p}_{k_\ell}(z)/\widetilde{q}_{k_\ell}(z) - f(z)| \leq \inf_{r\in\mathcal{I}_{mn}^{\Theta}} \|r - f\|_{K}.
\end{equation}
Thus, this lower bound holds true for all $z\in K$ except for at most $n$ points. 
Since we assume $K$ has no isolated points, continuity of the error in $K$ yields that the inequality~\eqref{eq:uerrornopoleproof} holds true for all $z\in K$. Thus,
\begin{equation*}
\|\widetilde{r}-f\| \leq \inf_{r\in\mathcal{I}_{mn}^{\Theta}} \|r - f\|_{K},
\end{equation*}
which shows that $\widetilde{r}$ attains the infimum~\eqref{eq:infinITK}.

It remains to show that $\widetilde{r}\in\mathcal{R}_{mn}$ attains $m+n+1$ interpolation nodes in $\Theta$ counting multiplicity, i.e., $\widetilde{r}\in\mathcal{I}_{mn}^{\Theta}$.
Since $\widetilde{r}_k\in\mathcal{I}_{mn}^{\Theta}$ for $k\in\mathbb{N}$, there exists a corresponding sequence of interpolation nodes $\{(z_{0,k},\ldots,z_{m+n,k})\}_{k\in\mathbb{N}}$ s.t.\ $\widetilde{r}_k$ is the Newton-Pad\'e approximant to $f$ for the nodes $z_{0,k},\ldots,z_{m+n,k}\in\Theta$. Since $\Theta$ is compact, we may choose a sub-sequence $k_i\subset k_\ell$ s.t.\ $z_{j,k_i}\to \widetilde{z}_j$ for $j=0,\ldots,m+n$ and nodes $ \widetilde{z}_0,\ldots, \widetilde{z}_{m+n}\in\Theta$. 
It remains to show that
\begin{equation}\label{eq:linsystemfabkj0}
(f \widetilde{q}-\widetilde{p})[\widetilde{z}_0,\ldots,\widetilde{z}_j] = 0,~~~j=0,\ldots,m+n.
\end{equation}
We recall that $\widetilde{q}_{k_i}\to \widetilde{q}$ and $\widetilde{p}_{k_i}\to \widetilde{q}$ coefficient-wise, and consequently $f \widetilde{q}_k-\widetilde{p}_k\to f \widetilde{q}-\widetilde{p}$ uniformly on any compact subset of $\mathbb{C}$. As noted in Proposition~\ref{prop:ddcont}, divided differences depend continuously on the underlying function as well as on the underlying nodes, and thus
\begin{equation*}
(f \widetilde{q}-\widetilde{p})[\widetilde{z}_0,\ldots,\widetilde{z}_{j}]
= \lim_{i \to \infty} (f \widetilde{q}_{k_i}-\widetilde{p}_{k_i})[z_{0,{k_i}},\ldots,z_{j, {k_i}}],~~~j=0,\ldots,m+n.
\end{equation*}
Since
\begin{equation*}
(f \widetilde{q}_{k_i}-\widetilde{p}_{k_i})[z_{0,{k_i}},\ldots,z_{j, {k_i}}] = 0,~~~j=0,\ldots,m+n,~~~i\in\mathbb{N}.
\end{equation*}
this implies that~\eqref{eq:linsystemfabkj0} holds true. Consequently, $\widetilde{r}$ is a Newton-Pad\'e approximant for the nodes $\widetilde{z}_0,\ldots,\widetilde{z}_{m+n}\in\Theta$, i.e., $\widetilde{r}\in\mathcal{I}_{mn}^{\Theta}$.
\end{proof}

We also consider interpolatory best approximations on shrinking domains. In particular, best approximants on $\varepsilon K$ with $m+n+1$ interpolation nodes in $\varepsilon \Theta$ and $\varepsilon\to 0$.

\begin{corollary}\label{prop:bestinterpconv}
Let $m$ and $n$ be given degrees, and assume the $(m,n)$-Pad\'e approximant to $f$ is non-degenerate. Then, interpolatory best approximants in $\mathcal{I}_{mn}^{\varepsilon\Theta}$ on $\varepsilon K$ exist exist for sufficiently small $\varepsilon$, i.e., there exist $r^\varepsilon\in\mathcal{I}_{mn}^{\varepsilon\Theta}$ with
\begin{equation}\label{eq:interpbestepsK}
\|r^\varepsilon - f\|_{\varepsilon K}
= \min_{r\in\mathcal{I}_{mn}^{\varepsilon\Theta}} \|r - f\|_{\varepsilon K}.
\end{equation}
Moreover, the interpolatory best approximations in $r^{\varepsilon}\in\mathcal{I}_{mn}^{\varepsilon\Theta}$ satisfy $r^\varepsilon \to r^P$ for $\varepsilon \to 0$.
\end{corollary}
\begin{proof}
{\em Existence.}
We first clarify that interpolatory best approximants in $\mathcal{I}_{mn}^{\varepsilon\Theta}$ exist for a sufficiently small $\varepsilon$. Let $\zeta_0,\ldots,\zeta_{m+n}\in\Theta$ denote an arbitrary but fixed sequence of points. Let $u^\varepsilon\in\mathcal{R}_{mn}$ denote the rational interpolant at the nodes $z_0,\ldots,z_{m+n}$ where $z_j(\varepsilon)=\varepsilon \zeta_j$ for $j=0,\ldots,m+n$. By construction $u^\varepsilon\in \mathcal{I}_{mn}^{\varepsilon\Theta}$. Let $\Delta_0$ denote a disk s.t.\ $r^P$ has no poles thereon. For a sufficiently small $\varepsilon$, this carries over to $r^\varepsilon$ as a consequence of Proposition~\ref{prop:cwconvergence}.\ref{item:repspoles}.
Moreover, for a sufficiently small $\varepsilon$ the set $\varepsilon K$ is contained in $\Delta_0$, and consequently, $u^\varepsilon$ has no poles on $\varepsilon K$ which implies $\|u^\varepsilon-f\|_{\varepsilon K}<\infty$. Thus, the conditions of Proposition~\ref{prop:bestinterp} hold true which shows existence of interpolatory best approximants $r^\varepsilon\in\mathcal{I}_{mn}^{\varepsilon\Theta}$, for sufficiently small $\varepsilon$. In particular, $r^\varepsilon$ satisfies~\eqref{eq:interpbestepsK}.

{\em Convergence.}
An interpolatory best approximant $r^\varepsilon \in \varepsilon\Theta$ attains interpolation nodes $z_j(\varepsilon)=\varepsilon \zeta_j(\varepsilon)$ with $\zeta_j(\varepsilon)\in\Theta$ for $j=0,\ldots,m+n$. Since $\Theta$ is compact the interpolation nodes satisfy $z_j(\varepsilon)\to 0$ for $\varepsilon\to 0$. Consequently, Proposition~\ref{prop:cwconvergence} implies $r^\varepsilon\to r^P$ for $\varepsilon \to 0$.
\end{proof}

\subsection{Examples}

Besides complex Chebyshev approximation, for which interpolation properties are discussed in the following section, interpolatory best approximations also cover real Chebyshev approximation to real functions and the unitary best approximation to the exponential function as shown in the following

\subsubsection{Real Chebyshev approximation to real functions}\label{subsec:realChebI}
We refer to the set of $(m,n)$-rational functions which are real-valued on $[-1,1]$ as $\mathcal{R}_{mn}^{\text{real}}$, and particularly, this set corresponds to the set of rational functions $r=p/q\in\mathcal{R}_{mn}$ where $p$ and $q$ have real coefficients up to a complex common factor of $p$ and $q$.

In general Chebyshev approximants to real functions $f$ on $K=[-1,1]$ are complex-valued~\cite[Chapter~24]{Tre13}. However, there is some interest on real rational Chebyshev approximants, i.e., best approximants in $\mathcal{R}_{mn}^{\text{real}}$ in a Chebyshev sense, in this setting due to their desirable properties, cf.~\cite[Theorem~24.1]{Tre13}. Following classical results stated therein, for the non-degenerate case the error of the real $(m,n)$-rational Chebyshev approximant $r^\text{real}$ to a real function $f$ on the interval $[-1,1]$ attains $m+n+2$ equioscillation points. Consequently, $r^\text{real}$ interpolates $f$ at $m+n+1$ nodes intermediate to the equioscillation points, i.e., $r^\text{real}\in\mathcal{I}_{mn}^{\Theta}$ for $K=[-1,1]$ and $\Theta=[-1,1]$. Moreover, since rational interpolants which interpolate real-valued data on the interval $[-1,1]$ are real-valued, i.e., $\mathcal{I}_{mn}^{\Theta}\subset \mathcal{R}_{mn}^{\text{real}}$, real Chebyshev approximant $r^\text{real}$ corresponds to the interpolatory best approximations~\eqref{eq:infinITK} in this setting. Analogously, this holds true when the underlying domains $K$ and $\Theta$ are scaled by $\varepsilon>0$ as in~\eqref{eq:interpbestepsK}.

\subsubsection{Unitary best approximation to the exponential function}\label{subsec:unitaryapproxI}

The unitary best approximation to $\exp(\mathrm{i}\omega x)$ for a given frequency $\omega>0$ and $x\in [-1,1]$ was introduced recently in~\cite{JS23u}. Rational approximation therein is subject to the notation $r(\mathrm{i} x) \approx \exp(\mathrm{i}\omega x)$, and can be rephrased as
\begin{equation}\label{eq:unitaryapproxexp}
r(z)\approx f(z)= \exp(z),~~~\text{ for $ z\in\varepsilon K $ and $K=\mathrm{i} [-1,1]$}.
\end{equation}
with $\varepsilon = \omega$ under consideration of re-scaling the argument of $r$. In particular, approximation~\eqref{eq:unitaryapproxexp} fits to the setting of the present work considering the scaling parameter $\varepsilon$ as the underlying frequency $\omega$. In the context of~\eqref{eq:unitaryapproxexp} and under the assumption that $\varepsilon$ is fixed and $n$ refers to a given degree, the {\em unitary best approximant} is defined as the $(n,n)$-rational approximant which minimizes $\|r- \exp\|_{\varepsilon K}$ in the set of unitary $(n,n)$-rational functions, where $r$ is referred to as unitary if
\begin{equation}\label{eq:defunitary}
|r(z)|=1~~~\text{for}~~z\in\mathrm{i}\mathbb{R}.
\end{equation}
This is in line with the definition of unitary best approximation to $\exp(\mathrm{i}\omega x)$ introduced in~\cite{JS23u}. In particular, existence and uniqueness of the unitary best approximant is discussed in sections~3 and~5 therein.

For rational approximation to the exponential function on a subset of the imaginary axis, unitarity is closely related to interpolation properties, and unitary best approximation can be understood as an interpolatory best approximation in general as shown in the following proposition.
\begin{proposition}\label{prop:unitaryandinterpolation}
Let $n$ denote a given degree. The following statements hold true.
\begin{enumerate}[label=(\roman*)]
\item\label{item:unitarprop} An $(n,n)$-rational function is unitary~\eqref{eq:defunitary} if and only if it interpolates the exponential function at $\geq 2n+1$ nodes on the imaginary axis, counting multiplicity.
\item\label{item:unitaryandinterpolatorybest} For~\eqref{eq:unitaryapproxexp} with $\varepsilon\in(0,(n+1)\pi)$, the unitary best approximation corresponds to the interpolatory best approximation~\eqref{eq:interpbestepsK} with $\Theta=\mathrm{i}[-1,1]$.
\end{enumerate}
\end{proposition}
\begin{proof} We first proof~\ref{item:unitarprop}.
Following \cite[Section~2]{JS24} (and~\cite[Subsection~2.1.3]{JS23u} for the case of osculatory interpolation nodes), $(n,n)$-rational functions which interpolate the exponential function at $2n+1$ nodes on the imaginary axis are unitary. On the other hand, unitary rational functions and the exponential function map the imaginary to the unit circle. Since $(n,n)$-rational functions are either constant or attain each point in the complex plain at most $n$ times, unitary rational functions rotate around the origin at most $n$ times when considered as a function of the imaginary axis. Since the exponential function as a function of the imaginary axis rotates around the origin an infinite number of times, this entails an infinite number of intersections of these two functions. Thus, unitary rational functions attain $\geq 2n+1$ interpolation nodes on the imaginary axis.

We proceed with the proof of~\ref{item:unitaryandinterpolatorybest}.
Following~\cite[Section~5]{JS23u}, particularly Corollary~5.2 therein, the unitary best approximant to $\exp(\mathrm{i} \omega x)$ attains $2n+1$ interpolation nodes on $[-1,1]$ for $\omega\in(0,(n+1)\pi)$. Rephrasing this result for the present setting, particularly $r\approx \exp$ as in~\eqref{eq:unitaryapproxexp} with $\varepsilon=\omega$, the unitary best approximant $r\approx\exp$ attains $2n+1$ interpolation nodes in $\mathrm{i} [-1,1]$. Consequently, this approximation is in the set $\mathcal{I}_{nn}^{\varepsilon \Theta}$ for $\Theta=\mathrm{i} [-1,1]$ which, following~\ref{item:unitarprop}, is a subset of the set of unitary $(n,n)$-rational functions. This implies that the unitary best approximation corresponds to the interpolatory best approximation~\eqref{eq:interpbestepsK} in the present setting.
\end{proof}

\section{Interpolation properties of complex Chebyshev approximation}\label{sec:chebipnodes}

In the following proposition we show that Chebyshev approximants on shrinking domains $\varepsilon K$ in a complex settings attain interpolation nodes on a disk, and thus, correspond to interpolatory best approximations for some $\Theta$ when $\varepsilon$ is sufficiently small.

\begin{proposition}\label{prop:Chebinterpolates}
Assume the $(m,n$)-Pad\'e approximant to $f$ is non-degenerate and $a_{mn}\neq 0$.
Let $r^\varepsilon\in\mathcal{R}_{mn}$ denote an $(m,n)$-rational Chebyshev approximant $r^\varepsilon\approx f$ on $\varepsilon K$, i.e.,
\begin{equation}\label{eq:minerrChebeps}
\|r^\varepsilon - f\|_{\varepsilon K}
= \min_{r\in\mathcal{R}_{mn}} \|r - f\|_{\varepsilon K}.
\end{equation}
For a sufficiently small $\varepsilon>0$ and a radius $\rho>0$ depending on $m$, $n$ and $K$, the difference $r^\varepsilon-f$ has exactly $m+n+1$ zeros in $\varepsilon\rho\Delta$ counting multiplicity.

In particular, Chebyshev approximation corresponds to interpolatory best approximation in $\mathcal{I}_{mn}^{\varepsilon \Theta}$ as in~\eqref{eq:interpbestepsK} for $\Theta=\rho\Delta$.
\end{proposition}
\begin{proof}
Following~\cite[Theorem~3b]{TG85},
Chebyshev approximants $r^\varepsilon$ to $f$ on $\varepsilon K$ converge to the Pad\'e approximant $r^P\approx f$, i.e., $r^\varepsilon\to r^P$ for $\varepsilon\to 0$.
While Chebyshev approximants are not necessarily unique, the notation $r^\varepsilon$ for $\varepsilon>0$ refers to an arbitrary but fixed sequence of Chebyshev approximants which satisfy~\eqref{eq:minerrChebeps} in the present proof.

Following Proposition~\ref{prop:zerosreps}, there exists a radius $\alpha>0$ s.t.\ the error $r^\varepsilon-f$ has exactly $m+n+1$ zeros $z_0(\varepsilon),\ldots,z_{m+n}(\varepsilon)$ on the disk $\alpha\Delta$ for sufficiently small $\varepsilon$. Moreover, Proposition~\ref{prop:zerosreps} shows that these zeros converge to the origin, i.e., $z_j(\varepsilon)\to 0$ for $\varepsilon \to 0$ and $j=0,\ldots,m+n$, and we may assume $z_j(\varepsilon)\in E$.
As noted in Remark~\ref{rmk:fromzerostonodes}, we may also refer to the zeros $z_0(\varepsilon),\ldots,z_{m+n}(\varepsilon)$ as interpolation nodes, where zeros of higher multiplicity can be understood as interpolation nodes of higher order. In particular, for a sufficiently small $\varepsilon>0$, the Chebyshev approximant $r^\varepsilon$ satisfies interpolation conditions~\eqref{eq:ratintcondconfluent} and corresponds to a Newton-Pad\'e approximant for the nodes $z_j(\varepsilon)$.

To proceed we define
\begin{equation*}
\zeta_j(\varepsilon) = z_j(\varepsilon)/\varepsilon,~~~j=0,\ldots,m+n,
\end{equation*}
and we let $\Delta_0$ denote a disk which contains no poles of $r^P$.

Since the Chebyshev approximant $r^\varepsilon$ can be understood as a rational interpolant with interpolation nodes $z_j(\varepsilon)\to0$, the results of Proposition~\ref{prop:cwconvergence} apply for $r^\varepsilon$.
In particular, for a sufficiently small $\varepsilon>0$ with $\varepsilon K\subset\Delta_0$ the Chebyshev approximant $r^\varepsilon$ satisfies the error representation~\eqref{eq:asymerrorinterp0}.
Substituting $\varepsilon \zeta_j(\varepsilon)$ and $\varepsilon z$ for $z_j(\varepsilon)$ and $z$, respectively, therein, we observe
\begin{equation}\label{eq:errinterpolate4}
r^\varepsilon(\varepsilon z) - f(\varepsilon z) = \varepsilon^{m+n+1} v^\varepsilon(\varepsilon z) \prod_{j=0}^{m+n} (z - \zeta_j(\varepsilon)) ,~~~z\in K.
\end{equation}
Since $r^\varepsilon$ is a Chebyshev approximant and minimizes the uniform error on $\varepsilon K$, the point-wise error of $r^\varepsilon$ is bounded by the uniform error of the rational interpolant at scaled Chebyshev nodes~\eqref{eq:errorlimitinterpchebnodes}. In particular,
\begin{equation*}
|r^\varepsilon(\varepsilon z) - f(\varepsilon z)|
\leq |a_{mn}|t_{m+n+1} \varepsilon^{m+n+1} + \mathcal{O}(\varepsilon^{m+n+2}),~~~\varepsilon \to 0,~~z\in K.
\end{equation*}
Substituting~\eqref{eq:errinterpolate4} therein and simplifying the identity, we observe
\begin{equation}\label{eq:errbound1x10}
\prod_{j=0}^{m+n} |z - \zeta_j(\varepsilon)| \left|v^\varepsilon(\varepsilon z)\right|
\leq |a_{mn}|t_{m+n+1} + \mathcal{O}(\varepsilon).
\end{equation}
We first assume the case $m+n>0$. Let $\kappa\in\{0,\ldots,m+n\}$ denote a fixed index.
Since $K$ is compact there exists $\chi_\kappa(\varepsilon)\in K$ s.t.
\begin{equation}\label{eq:errbound1x2}
\prod_{j\neq \kappa} |\chi_\kappa(\varepsilon) - \zeta_j(\varepsilon)|
= \max_{z\in K} \prod_{j\neq \kappa} |z - \zeta_j(\varepsilon)| \geq t_{m+n}>0.
\end{equation}
where $t_{m+n}$ refers to the Chebyshev constant of degree $m+n$ similar to~\eqref{eq:chebconst}.
Substituting $\chi_\kappa(\varepsilon)$ for $z$ in~\eqref{eq:errbound1x10} and making use of~\eqref{eq:errbound1x2}, we observe
\begin{equation}\label{eq:errbound1x3a}
|\chi_\kappa(\varepsilon) - \zeta_\kappa(\varepsilon)| \left|v^\varepsilon(\varepsilon \chi_\kappa(\varepsilon))\right|
\leq \frac{t_{m+n+1}}{t_{m+n}} |a_{mn}| + \mathcal{O}(\varepsilon).
\end{equation}

Following Proposition~\ref{prop:cwconvergence}.\ref{item:amncont0}, the remainder $v^\varepsilon$ satisfies 
$v^\varepsilon(\varepsilon z) \to a_{mn}$ for $\varepsilon\to 0$ and this convergence is uniform in $z$ on any compact set, particularly, for $z\in K$. This implies $v^\varepsilon(\varepsilon \chi_\kappa(\varepsilon))\to a_{mn}$ since $\chi_\kappa(\varepsilon)\in K$, and due to the assumption $a_{mn}\neq 0$ we conclude $a_{mn}/v^\varepsilon(\varepsilon \chi_\kappa(\varepsilon))\to 1$ for $\varepsilon\to 0$. In particular, for a sufficiently small $\varepsilon$, we have $|a_{mn}/v^\varepsilon(\varepsilon \chi_\kappa(\varepsilon))|< 2$ and~\eqref{eq:errbound1x3a} simplifies to
\begin{equation}\label{eq:chimzetabound}
|\chi_\kappa(\varepsilon) - \zeta_\kappa(\varepsilon)| \leq 2 \frac{t_{m+n+1}}{t_{m+n}},~~~\kappa=0,\ldots,m+n.
\end{equation}
The index $\kappa$ therein refers to an arbitrary and fixed index in $\{0,\ldots,m+n\}$.
Making use of~\eqref{eq:chimzetabound} and replacing the index $\kappa$ by $j$, we conclude that for a sufficiently small $\varepsilon$ the re-scaled interpolation nodes are bounded by
\begin{equation}\label{eq:zetaboundgeq0}
| \zeta_j(\varepsilon)|\leq |\chi_j(\varepsilon) - \zeta_j(\varepsilon)| + |\chi_j(\varepsilon)|\leq 2 \frac{t_{m+n+1}}{t_{m+n}} + \max_{z\in K} |z|,~~~j=0,\ldots,m+n.
\end{equation}
In a similar manner, for the trivial case $m=n=0$ the inequality~\eqref{eq:errbound1x10} and convergence  $a_{00}/v^\varepsilon(\varepsilon z )\to 1$ uniformly for $z\in K$ yields
\begin{equation*}
|z - \zeta_0(\varepsilon)|
\leq 2 t_{1},~~~\text{for $z\in K$ and sufficiently small $\varepsilon$},
\end{equation*}
and thus, for a fixed $z\in K$,
\begin{equation}\label{eq:zetaboundeq0}
|\zeta_0(\varepsilon)| \leq |z - \zeta_0(\varepsilon)| + |z|
\leq 2t_1 + \max_{z\in K} |z|,~~~m=n=0.
\end{equation}
Combining~\eqref{eq:zetaboundgeq0} and~\eqref{eq:zetaboundeq0}, we conclude that for $m,n\geq0$ and a sufficiently small $\varepsilon$
\begin{equation*}
\zeta_j(\varepsilon)\in\rho\Delta,~~~\text{for}~~\rho = 2 \frac{t_{m+n+1}}{t_{m+n}} + \max_{z\in K} |z|,~~~j=0,\ldots,m+n,
\end{equation*}
using the convention $t_0:=1$.
Thus, $r^\varepsilon$ attains interpolation nodes $z_j(\varepsilon)=\varepsilon \zeta_j(\varepsilon) \in \varepsilon \rho \Delta$, and for $\Theta = \rho \Delta$ this further implies $r^\varepsilon \in\mathcal{I}_{mn}^{\varepsilon \Theta}$.
Since the Chebyshev approximation minimizes the uniform error in $\mathcal{R}_{mn}$, it especially minimizes the uniform error in $\mathcal{I}_{mn}^{\varepsilon \Theta}\subset \mathcal{R}_{mn}$ which implies that a Chebyshev approximant $r^\varepsilon$ is also an interpolatory best approximant for $\Theta = \rho \Delta$.
\end{proof}

\section{Error of best approximations}\label{sec:asymerrorbest}

The following proposition applies to interpolatory best approximation, and most importantly, this includes the Chebyshev approximation as shown in Proposition~\ref{prop:Chebinterpolates}.

\begin{theorem}\label{thm:asymresults}
Let the $(m,n)$-Pad\'e approximant to $f:E\subset\mathbb{C}\to\mathbb{C}$ be non-degenerate and assume $a_{mn}\neq 0$.
Let $\Theta\subset E$ be a compact set which contains the $m+n+1$ Chebyshev nodes of $K$, i.e., $\tau_0,\ldots,\tau_{m+n}\in\Theta$.
Let $r^\varepsilon\in\mathcal{I}_{mn}^{\varepsilon\Theta}$ denote an interpolatory best approximation to $f$ on $\varepsilon K$, i.e., $r^\varepsilon$ satisfies~\eqref{eq:interpbestepsK}.
Then the following results hold true.
\begin{enumerate}[label=(\roman*)]
\item\label{item:cheb1} Let $z_0(\varepsilon),\ldots,z_{m+n}(\varepsilon)\in\varepsilon\Theta$ denote the interpolation nodes of $r^\varepsilon$ and define $\zeta_j(\varepsilon)=z_j(\varepsilon)/\varepsilon \in\Theta$ for $j=0,\ldots,m+n$. Then
\begin{equation}\label{eq:bestintnodesconv}
\zeta_j(\varepsilon) \to \tau_j,~~~j=0,\ldots,m+n,~~\varepsilon \to 0,
\end{equation}
up to ordering, where $\tau_j$ refer to the $m+n+1$ Chebyshev nodes of $K$.
\item\label{item:cheb2} Moreover, the points-wise error satisfies
\begin{equation}\label{eq:errorlimitbestiz}
r^{\varepsilon}(\varepsilon z)-f(\varepsilon z)
= a_{mn} \varepsilon^{m+n+1} \prod_{j=0}^{m+n}(z-\tau_j) + o(\varepsilon^{m+n+1}),~~~z\in K,
\end{equation}
for $\varepsilon\to 0$ with $a_{mn}$ as in~\eqref{eq:Padeasymerror}, and
\item\label{item:chebx} the uniform error satisfies
\begin{equation}\label{eq:errorlimitbesti}
\|r^\varepsilon-f\|_{\varepsilon K}
= t_{m+n+1} |a_{mn}| \varepsilon^{m+n+1}
+ \mathcal{O}(\varepsilon^{m+n+2}),
\end{equation}
for $\varepsilon\to 0$ with $t_{m+n+1}$ as in~\eqref{eq:chebconst}.
\end{enumerate}
\end{theorem}
\begin{proof}
Since the Pad\'e approximant is assumed to be non-degenerate, for a sufficiently small $\varepsilon$ the interpolatory best approximation $r^\varepsilon$ exists and converges to the Pad\'e approximation as noted in Corollary~\ref{prop:bestinterpconv}, i.e., $r^\varepsilon\to r^P$ for $\varepsilon\to 0$.

Following Proposition~\ref{prop:zerosreps}, the non-degenerate case with $a_{mn}\neq 0$ implies that, for some radius $\alpha>0$, the interpolatory best approximation $r^\varepsilon$ has exactly $m+n+1$ interpolation nodes in the disk $\alpha\Delta$ for sufficiently small $\varepsilon$. Thus, for a sufficiently small $\varepsilon$ s.t.\ $\varepsilon\Theta\subset \alpha\Delta$ we may refer to the interpolation nodes attained by $r^\varepsilon$ as $z_0(\varepsilon),\ldots,z_{m+n}(\varepsilon)\in\varepsilon\Theta$ with $z_j(\varepsilon)=\varepsilon \zeta_j(\varepsilon)$ where $\zeta_j(\varepsilon)\in\Theta$.

Proposition~\ref{prop:convinterp} provides an asymptotic error representation for rational interpolants with interpolation nodes $z_0(\varepsilon),\ldots,z_{m+n}(\varepsilon)$ with $z_j(\varepsilon)=\varepsilon \zeta_j$ where $\zeta_j \in\Theta$ are fixed. 
However, since the results of this proposition hold true uniformly for $\zeta_0,\ldots\zeta_{m+n}$ in the compact set $\Theta$, this carries over to the case of interpolation nodes depending on $\varepsilon$, i.e., $\zeta_j=\zeta_j(\varepsilon)\in\Theta$. In particular,~\eqref{eq:errorlimitinterpboth} implies
\begin{equation}\label{eq:asymerrinterinproof0}
r^\varepsilon(\varepsilon z) - f(\varepsilon z)
= a_{mn} \varepsilon^{m+n+1} \prod_{j=0}^{m+n}(z-\zeta_j(\varepsilon)) +\mathcal{O}(\varepsilon^{m+n+2}),~~~z\in K,
\end{equation}
and~\eqref{eq:errorlimitinterpuniformboth} implies
\begin{equation}\label{eq:asymerrinterinproof1}
\|r^\varepsilon - f \|_{\varepsilon K}
= |a_{mn}| \varepsilon^{m+n+1}  \max_{z\in K} \prod_{j=0}^{m+n}|z-\zeta_j(\varepsilon)| + \mathcal{O}(\varepsilon^{m+n+2}),
\end{equation}
for $\varepsilon \to 0$.
We remark that~\eqref{eq:asymerrinterinproof0} and~\eqref{eq:asymerrinterinproof1} hold true independently of convergence of $\zeta_j(\varepsilon)$.

We proceed with the proof of~\ref{item:chebx}.
Since we assume that $\Theta$ contains the Chebyshev nodes $\tau_0,\ldots,\tau_{m+n}\in \Theta$, the $(m,n)$-rational interpolant with interpolation nodes corresponding to the scaled Chebyshev nodes $\varepsilon\tau_0,\ldots,\varepsilon\tau_{m+n}$ is in the set $\mathcal{I}_{mn}^{\varepsilon\Theta}$. Thus, the error of the interpolatory best approximant $r^\varepsilon$ is bounded by the asymptotic error~\eqref{eq:errorlimitinterpchebnodes}, i.e., the error representation~\eqref{eq:errorlimitbesti} holds true as an upper bound,
\begin{equation}\label{eq:asymerrupperbound}
\|r^\varepsilon - f \|_{\varepsilon K}
\leq
t_{m+n+1} |a_{mn}| \varepsilon^{m+n+1} + \mathcal{O}(\varepsilon^{m+n+2}).
\end{equation}
On the other hand, substituting the lower bound from~\eqref{eq:chebconst} in~\eqref{eq:asymerrinterinproof1}, we observe
\begin{equation*}
\|r^\varepsilon - f \|_{\varepsilon K}
\geq t_{m+n+1}  |a_{mn}| \varepsilon^{m+n+1} + \mathcal{O}(\varepsilon^{m+n+2}).
\end{equation*}
Combining this with~\eqref{eq:asymerrupperbound}, we conclude~\eqref{eq:errorlimitbesti}.

We proceed with the proof~\ref{item:cheb1} by showing that the nodes $\zeta_0(\varepsilon),\ldots, \zeta_{m+n}(\varepsilon)$ attain a minimizing sequence for the $\infty$-norm~\eqref{eq:chebconst}. Namely, 
\begin{equation}\label{eq:minpolylikecheb}
\max_{z\in K} \prod_{j=0}^{m+n} |z-\zeta_j(\varepsilon)|
\to  \max_{z\in K} \prod_{j=0}^{m+n} |z-\tau_j|
= t_{m+n+1},~~~\text{for $\varepsilon\to0$}.
\end{equation}
We prove this claim by contradiction, assuming $\zeta_j(\varepsilon)$ is not a minimizing sequence of~\eqref{eq:minpolylikecheb} for $\varepsilon\to 0$. This implies that for some $\delta>0$ there exists a sequence $\{\varepsilon_\ell\}_{\ell\in\mathbb{N}}$ with $\varepsilon_\ell\to0$, s.t.
\begin{equation}\label{eq:minpolylikechebnot}
\max_{z\in K}\prod_{j=0}^{m+n} |z-\zeta_j(\varepsilon_\ell)|
\geq t_{m+n+1} + \delta,~~~\ell\in\mathbb{N}.
\end{equation}
Substituting $\varepsilon_\ell$, and consequently~\eqref{eq:minpolylikechebnot}, in~\eqref{eq:asymerrinterinproof1} shows
\begin{equation*}
\|r_{\varepsilon_\ell} - f \|_{\varepsilon_\ell K}
\geq |a_{mn}| \varepsilon_\ell^{m+n+1} ( t_{m+n+1} + \delta) + \mathcal{O}(\varepsilon_\ell^{m+n+2}),~~~\ell\to \infty.
\end{equation*}
However, this is a contradiction to~\eqref{eq:asymerrupperbound}, which implies that~\eqref{eq:minpolylikecheb} holds true.
We recall that the Chebyshev nodes uniquely minimize the right-hand side of~\eqref{eq:minpolylikecheb}. Since $\Theta$ is compact, this implies that $\zeta_j(\varepsilon)\to\tau_j$ up to ordering which completes the proof of~\ref{item:cheb1}.

To show~\eqref{eq:errorlimitbestiz} in~\ref{item:cheb2}, we apply~\eqref{eq:bestintnodesconv} from~\ref{item:cheb1} in~\eqref{eq:asymerrinterinproof0}. We remark that the remainder in~\eqref{eq:errorlimitbestiz} satisfies the convergence rate $o(\varepsilon^{m+n+1})$ since the presented proof of this error representation relies on $\zeta_j(\varepsilon)\to \tau_j$, i.e., $\zeta_j(\varepsilon)= \tau_j + o(1)$.
\end{proof}

\section{Applications}\label{sec:consequences}

\subsection{The error ratio of Chebyshev and Pad\'e approximants}\label{subsec:uniformerrorPadeCheb}

In the present subsection we provide an explicit formula for the quotient of the asymptotic errors of the Pad\'e approximant and interpolatory best approximants which follows Theorem~\ref{thm:asymresults}.\ref{item:chebx}. In the context of Chebyshev approximants, similar relations were previously discussed in~\cite[Section~3]{TG85}.
\begin{corollary}
In the setting of Theorem~\ref{thm:asymresults}, the uniform error of interpolatory best approximants $r^\varepsilon$ satisfies
\begin{equation*}
\lim_{\varepsilon \to 0}
\frac{\|r^P - f \|_{\varepsilon K}}{\|r^\varepsilon - f \|_{\varepsilon K} }
= \frac{\max_{z\in K}|z|^{m+n+1}}{t_{m+n+1}}.
\end{equation*}
\end{corollary}
\begin{proof}
This assertion follows from Theorem~\ref{thm:asymresults}.\ref{item:chebx} and
\begin{equation*}
\|r^P - f \|_{\varepsilon K} = |a_{mn}| \left(\max_{z\in K}|z| \varepsilon\right)^{m+n+1} + \mathcal{O}(\varepsilon^{m+n+2}),~~~\varepsilon\to 0.
\end{equation*}
\end{proof}

\subsection{Relation to previous convergence results}\label{subsec:relationtomeinardus}
Considering the notation of~\cite[Theorem~103]{Me67} for $a_{mn}$ in~\eqref{eq:amnformula}, i.e., $a_{mn} = (\Delta^{m+1,n+1}(f)/\Delta^{m,n}(f))_{x=0}$, we note that for $K=[-1,1]$ (and respectively, $t_{m+n+1}=2^{-(m+n)}$~\eqref{eq:ratinterpolateCheb}), the asymptotic error in Theorem~\ref{thm:asymresults}.\ref{item:chebx} reads
\begin{equation*}
\|r^\varepsilon-f\|_{\varepsilon K}
= \frac{\varepsilon^{m+n+1}}{2^{m+n}} \left|\left(\frac{\Delta^{m+1,n+1}(f)}{\Delta^{m,n}(f)}\right)_{x=0}\right| 
+ \mathcal{O}(\varepsilon^{m+n+2}),~~~\varepsilon\to 0.
\end{equation*}
Thus, for the interval $K=[-1,1]$ the asymptotically leading order term in Theorem~\ref{thm:asymresults}.\ref{item:chebx} coincides with the respective term in \cite[Theorem~103]{Me67}.

For the polynomial case $n=0$ the Pad\'e approximation corresponds to the Taylor approximation. Respectively, $a_{mn}$ in~\eqref{eq:Padeasymerror} for $n=0$ corresponds to
\begin{equation*}
a_{m0} =  f^{(m+1)}(0)/(m+1)!,~~~m\geq 0,
\end{equation*}
where $f^{(k)}(0)$ denotes the $k$th derivative of $f$ evaluated at $z=0$.
Consequently, Theorem~\ref{thm:asymresults}.\ref{item:chebx} for the polynomial Chebyshev approximation and $K=[-1,1]$ reads
\begin{equation*}
\|r^\varepsilon-f\|_{\varepsilon K}
= \varepsilon^{m+1} \frac{|f^{(m+1)}(0)|}{2^{m} (m+1)!} 
+ \mathcal{O}(\varepsilon^{m+2}),~~~\varepsilon\to 0.
\end{equation*}
Thus, for the polynomial case and $K=[-1,1]$ the asymptotically leading order term in Theorem~\ref{thm:asymresults}.\ref{item:chebx} coincides with the respective term in \cite[Theorem~62]{Me67} and similar results shown previously in~\cite{MW60a,Ni62} for the real setting.

\subsection{Complex and real Chebyshev approximations}\label{subsec:complexvsreal}

In general Chebyshev approximants to real functions $f$ on $K=[-1,1]$ are complex-valued~\cite[Chapter~24]{Tre13}, but there is some interest on real rational Chebyshev approximation in this setting as mentioned in Subsection~\ref{subsec:realChebI}.
Comparing the error of real and complex Chebyshev approximations to real functions is topic of~\cite{Le86} and others. Theorem~\ref{thm:asymresults} provides further insight on the asymptotic error of real and complex Chebyshev approximations when the underlying domain of convergence shrinks to the origin. In particular, real and complex Chebyshev approximations correspond to interpolatory best approximations~\eqref{eq:interpbestepsK} with interpolation nodes on the interval $\Theta=[-1,1]$ and the disk $\Theta=\rho\Delta$ for some radius $\rho$, respectively, as shown in Subsection~\ref{subsec:realChebI} and Section~\ref{sec:chebipnodes}. Thus, in both cases the convergence results in Theorem~\ref{thm:asymresults} apply for $\varepsilon\to 0$, showing that real and complex Chebyshev approximations attain the same asymptotic error as the underlying interval shrinks to the origin.

\subsection{Exponential function}\label{subsec:exp}

{\em Convergence on shrinking domains and in the limit $m+n \to \infty$.}
For the exponential function $f=\exp$ the error of the $(m,n)$-Chebyshev approximation on a disk with a fixed radius and increasing degree $m+n\to\infty$ was previously studied in~\cite{Tre84a} and an explicit formula for the asymptotic error is given therein. We remark that the asymptotically leading order term for $m+n\to\infty$ coincides with the respective term for the $(m,n)$-Chebyshev approximation for fixed degrees $m$ and $n$ on a shrinking disk $K=\Delta$, see Theorem~\ref{thm:asymresults}.\ref{item:chebx} with $a_{mn}$ as given in~\eqref{eq:amnexp} and $t_{m+n+1}=1$.
A similar relation holds true on the interval $K=[-1,1]$ with $t_{m+n+1}=2^{-(m+n)}$, for which the asymptotic error for $m+n\to \infty$ is provided in~\cite{Bra84}.

\smallskip
{\em Unitary best approximation.} The unitary best approximation to $\exp(\mathrm{i}\omega x)$, for a frequency $\omega>0$ and $x\in [-1,1]$, was introduced recently in~\cite{JS23u}, and can be rephrased to fit to the setting of the present work as shown in Subsection~\ref{subsec:unitaryapproxI}. Following Proposition~\ref{prop:unitaryandinterpolation}.\ref{item:unitaryandinterpolatorybest}, the unitary best approximation corresponds to the interpolatory best approximation to $f=\exp$ with $K=\Theta=\mathrm{i}[-1,1]$. Thus, convergence results for $\varepsilon\to 0$ in Theorem~\ref{thm:asymresults} apply with $a_{nn}$ as in~\eqref{eq:amnexp} and $t_{m+n+1}=2^{-(m+n)}$ (substituting $m=n$ therein). These results are in line with convergence results discussed in \cite[sections~8 and 9]{JS23u} for $\omega\to 0$, which is equivalent to $\varepsilon\to0$ in the present work.

Moreover, Theorem~\ref{thm:asymresults} also applies for the Chebyshev approximation without the restriction to unitarity in this setting, i.e., $r\approx \exp$ for $K=\mathrm{i}[-1,1]$, which correspond to interpolatory best approximations for $\Theta = \rho\Delta$ for some radius $\rho$ as discussed in Section~\ref{sec:chebipnodes}. Comparing the error of the Chebyshev approximation and the unitary best approximation to $\exp(\mathrm{i}\omega x)$ will be topic of a future work. In this context Theorem~\ref{thm:asymresults} shows that the asymptotic error of these approximations coincides for $\varepsilon\to 0$, or respectively, when the underlying frequency $\omega$ goes to zero.

\end{document}